\documentclass[]{siamart190516}
\usepackage{amsmath,amssymb}
\usepackage{blkarray}
\usepackage{graphicx}
\usepackage{mathdots}
\usepackage{subcaption}
\usepackage{xcolor}

\newcommand{\dd}{\mathrm{d}}
\newcommand{\ee}{\mathrm{e}}
\newcommand{\ii}{\boldsymbol{\mathrm{i}}}

\newcommand{\Z}{\mathbb{Z}}
\newcommand{\C}{\mathbb{C}}

\newcommand{\dime}{d}
\newcommand{\ind}{\ell}
\newcommand{\order}{n}

\newcommand{\ve}[1]{\boldsymbol{#1}}
\newcommand{\npr}{d_{\ven}}
\newcommand{\vej}{\ve{j}}
\newcommand{\vek}{\ve{k}}

\newcommand{\ven}{\ve{n}}
\newcommand{\p}{_{\ven}}
\newcommand{\vet}{\ve{\theta}}
\newcommand{\veq}{\ve{q}}
\newcommand{\vex}{\ve{x}}
\newcommand{\sn}{_{\ven}}

\newcommand{\sk}{_{\vek}}

\newcommand{\Y}{Y}

\newcommand{\Tn}{T\sn}
\newcommand{\Pin}{\Pi\sn}
\newcommand{\Yn}{Y\sn}
\newcommand{\Un}{U\sn}

\newcommand{\Prec}{\mathcal{P}}

\newcommand{\Id}{[-\pi,\pi]^\dime}
\newcommand{\glteq}{\sim_{\rm GLT}}

\newtheorem{remark}{Remark}
\theoremstyle{plain}
\newtheorem{example}{Example}

\graphicspath{{./matlab_tests/plots/}}

\title{The asymptotic spectrum of flipped multilevel Toeplitz matrices and of certain preconditionings}
\author{M.\ Mazza$^1$ and J.\ Pestana$^{2,}$\thanks{corresponding author
\newline
${}^1$ Department of Humanities and Innovation, University of Insubria, via Valleggio 11, 22100 Como, Italy (\texttt{mariarosa.mazza@uninsubria.it}),
		\newline
${}^2$	Department of Mathematics and Statistics, University of Strathclyde, 16 Richmond Street,
	Glasgow G1 1XQ, UK (\texttt{jennifer.pestana@strath.ac.uk})}}

\begin{document}
\maketitle

\begin{abstract}
In this work, we perform a spectral analysis of flipped multilevel Toeplitz sequences, i.e., we study the asymptotic spectral behaviour of $\{\Yn\Tn(f)\}\sn$, where $\Tn(f)$ is a real, square multilevel Toeplitz matrix generated by a function $f\in L^1([-\pi,\pi]^d)$ and
$\Yn$ is the exchange matrix, which has $1$s on the main anti-diagonal. In line with what we have shown for unilevel flipped Toeplitz matrix sequences, 
the asymptotic spectrum is determined by a $2\times 2$ matrix-valued function whose eigenvalues are $\pm |f|$. Furthermore, we characterize the eigenvalue distribution of certain preconditioned flipped multilevel Toeplitz sequences with an analysis that covers both multilevel Toeplitz and circulant preconditioners. Finally, all our findings are illustrated by several numerical experiments.

\end{abstract}



\section{Introduction}
In \cite{FFHM19,MaPe19} it was independently shown that when $T_n(f) \in \mathbb{R}^{n\times n}$ is a Toeplitz matrix generated by a function $f \in L^1([-\pi,\pi])$ then 
the eigenvalues of $Y_nT_n(f)$ are distributed like $\pm |f|$, where 
\begin{equation*}
\label{eqn:flip1d}
Y_n = 
\begin{bmatrix}
& & 1\\
& \iddots & \\
1 & & 
\end{bmatrix}.
\end{equation*}
In this note, we show that this result also holds true when the Toeplitz matrix $T_n(f)$ is replaced by a  multilevel Toeplitz matrix $\Tn(f)\in\mathbb{R}^{\npr\times \npr}$ and $Y_n$ is replaced by 
\begin{equation*}
\label{eqn:flip}
\Yn=Y_{n_1}\otimes\cdots\otimes Y_{n_d}, 
\end{equation*}
where $f \in L^1([-\pi,\pi]^d)$, $\ven=(n_1,\ldots,n_d)$ and $\npr = n_1\dotsm n_\dime$. 
More specifically, we prove that the eigenvalues of $\Yn\Tn(f)$ behave like the eigenvalues of the matrix-valued symbol 
\begin{equation}
\label{eqn:g}
g = 
\begin{bmatrix}
& f\\
f^\ast & 
\end{bmatrix},
\end{equation}
where $f^\ast$ is the conjugate of $f$, i.e., the eigenvalues of the flipped multilevel Toeplitz matrix $\Yn\Tn(f)$ are distributed like $\pm |f|$. 

Describing the spectra of these flipped matrices is important for solving linear systems with $\Tn(f)$ as coefficient matrix. Since a (multilevel) Toeplitz matrix can  be symmetrized by the flip matrix, the resulting linear system may be solved by MINRES or preconditioned MINRES, with its short term recurrences and descriptive convergence theory based on eigenvalues~\cite{Pest19,PeWa15}. Hence, knowledge of the spectrum of $\Yn\Tn(f)$ is critical for accurately estimating the MINRES convergence rate, and developing and analysing effective preconditioners. With this in mind, we characterize the eigenvalue distribution of certain preconditioned flipped multilevel Toeplitz sequences with an analysis that covers both multilevel Toeplitz and circulant preconditioners.

The paper is organized as follows. \Cref{sec:prelim} provides background material and preliminary results. The key results are then presented in \cref{sec:main}, and illustrated by numerical experiments in \cref{sec:num}. Our conclusions can be found in \cref{sec:conc}.

\section{Preliminaries} \label{sec:prelim}
\begin{sloppypar}
In this section we formalize the definition of multilevel (block) Toeplitz sequences associated with a Lebesgue integrable (matrix-valued) function. Next, we define the spectral distribution, in the sense of the eigenvalues and of the singular values, of a generic matrix sequence. 
To deal with the spectral distribution of preconditioned flipped multilevel Toeplitz matrices, we introduce a class of matrix sequences (the multilevel block GLT class) that contains multilevel block Toeplitz sequences. 
\end{sloppypar}

\subsection{Notation}
\label{subsec:notation}
\begin{sloppypar}
To describe multilevel matrices we require multi-indices, $\vek = (k_1,\dotsc,k_\dime)$, that we denote  by bold letters. 
Whenever we use the expression  ${{\vek}\to\infty}$,  
we mean that every component of the vector ${\vek}$ tends to infinity, that is,
	$\min_{\ell=1,\dots,d} k_\ell\to\infty$.
\end{sloppypar}
The complex conjugates of a scalar $\alpha$, and scalar-valued function $f(\vet)$, 
where  $\vet = (\theta_1, \theta_2,\dotsc, \theta_\dime)$,   $\dime\ge 1$, are denoted by $\alpha^\ast$ and $f^\ast(\vet)$, respectively. 
Similarly, the conjugate transpose of a vector $x$ is $x^\ast$, and the conjugate transpose of a matrix $X$ is $X^\ast$. 
Additionally, by $|f(\vet)|$ we mean $|f(\vet)|=\left(f(\vet)f^*(\vet)\right)^{1/2}$. 
The $n\times n$ identity matrix is $I_n$. 
 
Throughout, by $d$-level $s\times s$-block matrix sequences we mean sequences of matrices of the form $\{A\sn\}_n$, where the index $n$ varies in an infinite subset of $\mathbb{N}$ and $\ven = \ven(n)$ is a $d$-index with positive components that depends on $n$ and satisfies $\ven\rightarrow\infty$ as $n\rightarrow\infty$. The size of $A\sn$ is $s\cdot d\sn=sn_1\cdots n_d$. We will equivalently use the notation \lq\lq $\{A\sn\}\sn$'' to mean \lq\lq $\{A\sn\}_n$''.

 \subsection{Multilevel block Toeplitz matrices and their spectral properties}
In \cref{def:Tml} we introduce the notion of multilevel block Toeplitz matrix sequences generated by $f$.
\begin{definition}\label{def:Tml}
	Let $f:[-\pi, \pi]^d\rightarrow \mathbb{C}^{s\times s}$, $d,s\ge1$, where $f=[f_{ij}]_{i,j=1}^s$ is  such that $f_{ij}\in L^1([-\pi,\pi]^d)$. Let the Fourier coefficients of $f$ be given by 
	\[t\sk:=\frac{1}{(2\pi)^\dime}\int_{\Id} f(\vet) \, \ee^{-\ii \langle \vek, 
\vet\rangle}\, \dd\vet\in\mathbb{C}^{s\times s},\quad \vek = (k_1,\dotsc, k_d) \in \Z^\dime,\]
where the integrals are computed componentwise and $\langle \vek,\vet\rangle = \sum_{\ind=1}^\dime k_\ind \theta_\ind$. 
The $\ven$-th $d$-level $s\times s$-block Toeplitz matrix associated with $f$ is the matrix of order $s\cdot\npr$,  $\npr=n_1\dotsm n_\dime$, given by
\[\Tn(f) = \sum_{|k_1| < n_1} \dotsb \sum_{|k_\dime| < n_\dime} J_{n_1}^{(k_1)} \otimes \dotsb \otimes J_{n_\dime}^{(k_\dime)}\otimes t\sk,\]
where $J_\order^{(\ind)}$ is the matrix of dimension $\order$ whose ($i,j$) entry is 1 if $i-j = \ind$ and is zero otherwise. 
	The set $\{\Tn(f)\}\sn$ is
	called the \emph{family of multilevel block Toeplitz matrices generated by $f$}. 
	The function $f$ is referred to as the \emph{generating function}  of $\{\Tn(f)\}\sn$.
\end{definition}

We now discuss the spectra of multilevel block Toeplitz matrices. 
To clarify the sense in which the function $f$ provides information on the spectrum for these problems, we need to introduce the following definition.

\begin{definition}\label{def-distribution}
	Let $f:G\to\mathbb{R}^{s\times s}$ be a measurable function, defined on a measurable set $G\subset\mathbb R^d$ with $d\ge 1$,
	$0<m_\ell(G)<\infty$. Let $\mathcal C_0(\mathbb K)$ be the set of continuous functions with compact support over $\mathbb
	K\in \{\mathbb C, \mathbb R_0^+\}$ and let $\{A\sn\}\sn$, ${\ven}\in\mathbb{N}^v$ $v\ge 1$, be a sequence of matrices with
	eigenvalues $\lambda_j(A\sn)$, $j=1,\ldots,d\sn$ and singular
	values $\sigma_j(A\sn)$, $j=1,\ldots,
	d\sn$, where $d\sn={\rm dim}(A\sn)$ is a monotonic function with respect to each variable $n_i$, $i=1,\ldots,v$.
	
	\begin{itemize}
		\item We say that $\{A\sn\}\sn$ is {\em distributed as the pair
			$(f,G)$ in the sense of the eigenvalues,} and we write $\{A\sn\}\sn\sim_\lambda(f,G),$ if the following limit relation holds for all $F\in\mathcal C_0(\mathbb C)$:
		\begin{align}\label{distribution:eig}
		\lim_{{\ven}\to\infty}\frac{1}{d\sn}\sum_{j=1}^{d\sn}F(\lambda_j(A\sn))=
		\frac1{m_\ell(G)}\int_G \frac{{\rm tr}(F(f(\vet)))}{s} d\vet.
		\end{align}
		In this case, we say that f is the \emph{symbol} of the matrix sequence $\{A\sn\}\sn$.
		\item We say that $\{A\sn\}\sn$ is {\em distributed as the pair
			$(f,G)$ in the sense of the singular values,} and we write $\{A\sn\}\sn\sim_\sigma(f,G),$ if the following
		limit relation holds for all $F\in\mathcal C_0(\mathbb R_0^+)$:
		\begin{align}\label{distribution:sv}
		\lim_{{\ven}\to\infty}\frac{1}{d\sn}\sum_{j=1}^{d\sn}F(\sigma_j(A\sn))=
		\frac1{m_\ell(G)}\int_G \frac{{\rm tr}(F(|f(\vet)|))}{s} d\vet.
		\end{align}
	\end{itemize}
	Recall that in this setting the expression ${{\ven}\to\infty}$ means that every component of the vector ${\ven}$ tends to infinity, that is,
	$\min_{i=1,\dots,v} n_i\to\infty$. 
\end{definition}

\begin{remark}\label{rem:approx}
	If $f$ is smooth enough, an informal interpretation of the limit relation
	\cref{distribution:eig} (resp. \cref{distribution:sv})
	is that when $\ven$ is sufficiently large,
	 $d\sn/s$ eigenvalues (resp.\ singular values) of $A\sn$ can
	be approximated by a sampling of $\lambda_1(f)$ (resp.\ $\sigma_1(f)$)
	on a uniform equispaced grid of the domain $G$, and so on until the
	last $d\sn/s$ eigenvalues (resp.\ singular values), which can be approximated by an equispaced sampling
	of $\lambda_s(f)$ (resp. $\sigma_s(f)$) in the domain.
\end{remark}

The above definitions are applicable to multilevel Toeplitz matrix sequences, 
as the following theorem (due to Szeg\H{o}, Tilli, Zamarashkin, Tyrtyshnikov, ...) shows.
\begin{theorem}[see \cite{GSz,Tillinota,TyL1}]\label{szego}
	Let $\{T\sn(f)\}\sn$ be a multilevel Toeplitz sequence generated by $f\in L^1([-\pi,\pi]^d)$. Then, $\left\{T\sn(f)\right\}\sn\sim_\sigma (f,[-\pi,\pi]^d).$
	Moreover, if $f$ is real-valued, then 
	$\left\{\Tn(f)\right\}\sn\sim_\lambda (f,[-\pi,\pi]^d).$
\end{theorem}

In the case that $f$ is a Hermitian matrix-valued function, the previous theorem can be extended as follows:
\begin{theorem}[see \cite{Tillinota}]\label{szego-herm}
	Let $f:[-\pi, \pi]^d\rightarrow \mathbb{C}^{s\times s}$, $d>1$, with $f=[f_{ij}]_{i,j=1}^s$ such that $f_{ij}\in L^1([-\pi,\pi]^d)$, be a Hermitian matrix-valued function. Then, $\{\Tn(f)\}\sn\sim_\lambda(f,[-\pi, \pi]^d).$
\end{theorem}

The following theorem is a useful tool for computing the spectral distribution of a sequence of Hermitian matrices. For its proof, see \cite[Theorem 4.3]{mari}. 

\begin{theorem}\label{th:extradimensional}
	Let $f:G\subset\mathbb{R}^d\rightarrow\mathbb{C}^{s\times s}$, let $\{X\sn\}\sn$ be a sequence of matrices with $X\sn$ Hermitian of size $d\sn$, and let $\{P\sn\}\sn$ be a sequence such that $P\sn\in\mathbb C^{d\sn\times\delta\sn}$, $P\sn^*P\sn=I_{\delta\sn}$, $\delta\sn\le d\sn$ and $\delta\sn/d\sn\to1$ as $\ven\to\infty$. Then $\{X\sn\}\sn\sim_{\lambda}(f,G)$ if and only if $\{P\sn^*X\sn P\sn\}\sn\sim_{\lambda}(f,G)$.
\end{theorem}

\subsection{Multilevel block generalized locally Toeplitz class}
In the sequel, we introduce the $*$-algebra of multilevel block generalized locally Toeplitz (GLT) matrix sequences \cite{axioms,GaSe17}. The formal definition of this class is rather technical and involves somewhat cumbersome notation: therefore we just give and briefly discuss a few properties of the multilevel block GLT class, which are sufficient for studying the spectral features of preconditioned flipped multilevel Toeplitz matrices.

Throughout, we use the notation
$$\{A\sn\}\p\sim_{\rm GLT} { \kappa(\vex,\vet)},\quad \ \ \kappa:[0,1]^d\times[-\pi,\pi]^d\rightarrow\mathbb{C}^{s\times s}$$
to indicate that the sequence $\{A\sn\}\p$ is a $d$-level $s\times s$-block GLT sequence with GLT symbol $\kappa(\vex,\vet)$.

Here we list five of the main features of multilevel block GLT sequences.
\begin{itemize}
	\item[{\bf GLT1}] Let $\{A\p\}\p\sim_{\rm GLT}\kappa$ with $\kappa:G\rightarrow \mathbb{C}^{s\times s}$, $G=[0,1]^d\times[-\pi,\pi]^d$. Then $\{A\p\}\p\sim_\sigma(\kappa,G)$. If the matrices $A\p$ are Hermitian, then it also holds that $\{A\p\}\p\sim_\lambda(\kappa,G)$.
	\item[{\bf GLT2}] The set of block GLT sequences forms a $*$-algebra, i.e., it is closed under linear combinations, products, inversion and conjugation. In formulae, let $\{ A\p \}\p \glteq \kappa_1$ and $\{ B\p \}\p \glteq \kappa_2$, then
	\begin{itemize}
		
		\item[$\bullet$] $\{\alpha A\p + \beta B\p\}\p \glteq \alpha\kappa_1+\beta \kappa_2, \quad \alpha, \beta \in \mathbb{C};$
		\item[$\bullet$] $\{A\p B\p\}\p \glteq \kappa_1 \kappa_2;$
		\item[$\bullet$] $\{A^{-1}\p\}\p \glteq \kappa_1^{-1}$ provided that $\kappa_1$ is invertible a.e.;
		\item[$\bullet$] $\{ A\p^{*} \}\p \glteq {\kappa^*_1}.$
	\end{itemize}
	\item[{\bf GLT3}] Any sequence of multilevel block Toeplitz matrices $\{ T\p(f) \}\p$ generated by a function $f:[-\pi, \pi]^d\rightarrow \mathbb{C}^{s\times s}$, with $f=[f_{ij}]_{i,j=1}^s$ such that $f_{ij}\in L^1([-\pi,\pi]^d)$,  is a $d$-level $s\times s$-block GLT sequence with symbol $\kappa(\vex, \vet) = f(\vet)$.
	\item[{\bf GLT4}] Let $\{A\p\}\p\sim_\sigma 0$. We say that $\{A\p\}\p$ is a \emph{zero-distributed matrix sequence}. Note that for any $s>1$ $\{A\p\}\p\sim_\sigma O_s$, with $O_s$ the $s\times s$ null matrix, is equivalent to $\{A\p\}\p\sim_\sigma 0$. Every zero-distributed matrix sequence is a block GLT sequence with symbol $O_s$ and viceversa, i.e., $\{A\p\}\p\sim_\sigma0$ $\iff$ $\{A\p\}\p\sim_{\rm GLT}O_s$.
	\item[\textbf{GLT5}] Let $\{A\sn\}\sn$ be a $d$-level matrix sequence and let $\{B_{\ven,m}\}\sn$ be a sequence of matrix sequences that satisfies the following condition: for each $m$ there exists $n_m$, such that for $m>n_m$ 
	\begin{equation*}
	A_{\ven}=B_{\ven,m}+R_{\ven,m}+E_{\ven,m}
	\end{equation*}
	with 
	\begin{equation*}
	{\rm rank}(R_{\ven,m})<c(m)d_{\ven}, \quad \|E_{\ven,m}\|\le\omega(m), \quad \lim_{m\rightarrow\infty}c(m),\omega(m)=0. 
	\end{equation*}
	\begin{sloppypar}
	We say that $\{B_{\ven,m}\}$ is an \lq\lq a.c.s.'' (approximating class of sequences) for $\{A_{\ven}\}\sn$, and we write $\{B_{\ven,m}\}\sn\xrightarrow{a.c.s.}\{A_{\ven}\}\sn$. Moreover, $\{A\sn\}\sn\sim_{\rm GLT}\kappa$ if and only if there exist GLT sequences $\{B_{\ven,m}\}\sn\sim_{\rm GLT}\kappa_m$
	and $\kappa_m\rightarrow \kappa$ in measure.
	\end{sloppypar}
\end{itemize}

The following proposition provides an a.c.s.\ for a sequence of multilevel block Toeplitz matrices (see \cite{GaSe18}).

\begin{proposition}\label{prop:acs_toeplitz}
\begin{sloppypar}
Let $\{f_m\}_m$ be a sequence of d-variate trigonometric matrix-valued polynomials with $f_m:[-\pi, \pi]^d\rightarrow \mathbb{C}^{s\times s}$, $f_m=[(f_m)_{ij}]_{i,j=1}^s$ such that $(f_m)_{ij}\in L^1([-\pi,\pi]^d)$. If $(f_m)_{ij} \rightarrow (f)_{ij}$ in $L^1([-\pi,\pi]^d)$, then the sequence $\{\Tn(f_m)\}\sn$ satisfies
\begin{equation*}\label{eq:acs_toeplitz}
\{\Tn(f_m)\}\sn{\xrightarrow{a.c.s.}}\{\Tn(f)\}\sn.
\end{equation*}
\end{sloppypar}
\end{proposition}

We also give an additional characterization of zero-distributed matrix sequences that will prove useful:.
\begin{theorem}[see {\cite[Theorem 2.2]{GaSe18}}]
\label{thm:zero_dist}
Let $\{A\p\}\p$ be a sequence of matrices with $A\p$ of dimension $\npr$. Then $\{A\p\}\p\sim_{\sigma} 0$ if and only if, for every $\ven$, 
\[A\p = R\p + E\p, \qquad \lim_{\ven \to \infty} \frac{{\rm rank}(R\p)}{d_{\ven}} = 0, \qquad \lim_{\ven \to \infty} \|E\p\| = 0.\]
\end{theorem}

We next recall a result on the spectral distribution of Hankel sequences associated with $f:[-\pi, \pi]^d\rightarrow \mathbb{C}^{s\times s}$, where $f=[f_{ij}]_{i,j=1}^s$ is such that $f_{ij}\in L^1([-\pi,\pi]^d)$.

\begin{theorem}[see \cite{FaTi00}] \label{lem:hankel-zero-symbol}
Let $\{ H\sn(f) \}\sn$ be the $\ven$-th $s\times s$-block multilevel Hankel matrix associated with $f:[-\pi, \pi]^d\rightarrow \mathbb{C}^{s\times s}$, where $f=[f_{ij}]_{i,j=1}^s$ is such that $f_{ij}\in L^1([-\pi,\pi]^d)$. If $H\sn(f)$ is the $sd\sn\times sd\sn$ matrix
\begin{align*}
H_{\ven}(f)=\left[t_{\ve{i}+\vej-\ve{ 2}}\right]_{\ve{i},\vej=\ve{1}}^{\ven},
\end{align*}
with $t_{\vek}\in\C^{s\times s}$ the Fourier coefficients of $f$, then $\{ H\sn(f) \}\sn \sim_\sigma 0$. 
\end{theorem}

\begin{remark}
Note that one can equivalently  take $H_{\ven}(f)=\left[t_{\ve{ 2}-\ve{i}-\vej}\right]_{\ve{i},\vej=\ve{1}}^{\ven}$ in \cref{lem:hankel-zero-symbol}.
\end{remark}

Together, \cref{lem:hankel-zero-symbol} and {\bf GLT4}  tell us that $\{ H\sn(f) \}\sn$ is an $s\times s$-block GLT sequence with symbol $O_s$.

	We end this subsection with a theorem that is very useful in the context of preconditioning involving GLT matrix sequences. It is obtained as a straightforward extension of Theorem 1 in \cite{garoni2018} to the multilevel block GLT case, provided the symbol of the preconditioning sequence is a multiple of the identity.

\begin{theorem}\label{thm:prec}
	\begin{sloppypar}
	Let $\{A\sn\}\sn$ be a sequence of Hermitian matrices such that $\{A\sn\}\sn\sim_{\rm GLT}\kappa$, with $\kappa:G\rightarrow \mathbb{C}^{s\times s}$, $G=[0,1]^d\times[-\pi,\pi]^d$, and let $\{\Prec\sn\}\sn$ be a sequence of Hermitian positive definite matrices such that $\{\Prec\sn\}\sn\sim_{\rm GLT}h\cdot I_s$, with $h:G\rightarrow \mathbb{C}$, such that $h\ne0$ a.e. Then,
	\begin{equation*}
	\{\Prec\sn^{-1}A\sn\}\sn\sim_{\sigma,\lambda}(h^{-1}\kappa,G).
	\end{equation*}
	\end{sloppypar}
\end{theorem}

\section{Main result} \label{sec:main}
\begin{sloppypar}
In this section we prove the main result, namely that $\{\Yn\Tn(f)\}\sn {\sim_{\lambda}} (g,[-\pi,\pi]^d)$, 
where $g$ is given in \cref{eqn:g} and the dimension of $\Tn(f)$ is given by the multi-index $\ven = (n_1,\dotsc, n_d)$. 
\end{sloppypar}
We first introduce the following matrices:

\begin{itemize}
	\item $\Pi_{\ven}=\Pi_{n_1} \otimes \Pi_{n_2}\otimes \cdots \otimes \Pi_{n_d} $ with $\Pi_{n_k}$,  $n_k$ even, such that its $j$-th column $\pi_j$, $j = 1,\dotsc, n_k$, is 
	\begin{equation*}
	\label{eqn:perm}
	\pi_j = 
	\begin{cases}
	e_{2j-1}, & j=1,\dotsc, n_k/2,\\
	e_{2(j-n_k/2)}, & j= n_k/2+1,\dotsc, n_k,
	\end{cases}
	\end{equation*}
	where $e_j$, $j = 1,\dotsc, n_k$, is the $j$-th column of the identity matrix of dimension $n_k$;
		\item $\Yn=Y_{n_1} \otimes Y_{n_2}\otimes \cdots \otimes Y_{n_d} $ with $Y_{n_k}$ defined as
	\begin{equation*}
	Y_{n_k} = 
	\begin{bmatrix}
	& & 1\\
	& \iddots & \\
	1 & & 
	\end{bmatrix}_{n_k\times n_k};
	\end{equation*}
	\item $U_{\ven}=U_{n_1} \otimes U_{n_2}\otimes \cdots \otimes U_{n_d} $ with $U_{n_k}$ such that
	\begin{equation*}
	\label{eqn:un}
	U_{n_k} = \begin{bmatrix}  Y_{\lceil n_k/2\rceil} &\\ & I_{\lfloor n_k/2\rfloor} \end{bmatrix}.
	\end{equation*}
\end{itemize}

We now state an important preliminary result. 

\begin{proposition}
\label{lem:Y_GLT}
Assume that $\ven=(n_1,\ldots,n_d)$ with $n_k=2m_k$, $m_k \in \mathbb{N}$. Then, for any $f\in L^1([-\pi,\pi]^d])$, 
\[\{\Pin\Un\Yn\Tn(f)\Un\Pin^T\}\sn \sim_{\rm GLT} (g, [-\pi,\pi]^d)\text{\,\,\, with \,\,\,} 
g:=\begin{bmatrix} 0 & f\\ f^\ast & 0\end{bmatrix}.
\]
\end{proposition}
\begin{proof}
Let us first assume that $f(\vet)=f_{\veq}(\vet)=\sum_{\vej=-\veq}^{\veq}t_{\vej}{\rm e}^{\ii\langle\vej,\vet\rangle}$. Then,

\begin{align*}
\Pin\Un\Yn\Tn(f_{\veq})\Un\Pin^T&\\
&\hspace{-2.4cm}=\Pin\Un\Yn\left(\sum_{\vej=-\veq}^{\veq}t_{\vej}\Tn\left(\prod_{k=1}^d{\rm e}^{\ii j_k\theta_k}\right)\right)\Un\Pin^T\\
&\hspace{-2.4cm}=\Pin\Un\Yn\left(\sum_{\vej=-\veq}^{\veq}t_{\vej}T_{n_1}({\rm e}^{\ii j_1\theta_1})\otimes \cdots \otimes T_{n_d}({\rm e}^{\ii j_d\theta_d})\right)\Un\Pin^T\\
&\hspace{-2.4cm}=\sum_{\vej=-\veq}^{\veq}t_{\vej}\,\,\Pi_{n_1}U_{n_1}Y_{n_1}T_{n_1}({\rm e}^{\ii j_1\theta_1})U_{n_1}\Pi_{n_1}^T\otimes \cdots \otimes \Pi_{n_d}U_{n_d}Y_{n_d}T_{n_d}({\rm e}^{\ii j_d\theta_d})U_{n_d}\Pi^T_{n_d}.
\end{align*}
Now, by using Lemmas 3.1 and 3.2 in \cite{MaPe19} applied to $f(\theta_k)={\rm e}^{\ii j_k\theta_k}$ we find that
\begin{align*}
\Pi_{n_k}U_{n_k}Y_{n_k}T_{n_k}({\rm e}^{\ii j_k\theta_k})U_{n_k}\Pi_{n_k}^T&=T_{n_k}\left(\begin{bmatrix} 0 & {\rm e}^{\ii j_k\theta_k}\\ {\rm e}^{-\ii j_k\theta_k} & 0\end{bmatrix}\right)+ E_{n_k}+R_{n_k}
\end{align*}
with 
\[\lim_{n_k\rightarrow\infty} \frac{{\rm rank}(R_{n_k})}{n_k}=0, \,\, \lim_{n_k\rightarrow\infty}\|E_{n_k}\|=0.\]
Therefore,
\begin{align}\label{eq:pol}
\notag \Pin\Un\Yn\Tn(f_{\veq})\Un\Pin^T & \\
&\notag\hspace{-2.5cm}=\sum_{\vej=-\veq}^{\veq}t_{\vej}\,\,\left(T_{n_1}\left(\begin{bmatrix} 0 & {\rm e}^{\ii j_1\theta_1}\\ {\rm e}^{-\ii j_1\theta_1} & 0\end{bmatrix}\right)+ E_{n_1}+R_{n_1}\right)\otimes \cdots \\
&\notag\hspace{-0.5cm}\cdots\otimes \left(T_{n_d}\left(\begin{bmatrix} 0 & {\rm e}^{\ii j_d\theta_d}\\ {\rm e}^{-\ii j_d\theta_d} & 0\end{bmatrix}\right)+ E_{n_d}+R_{n_d}\right)\\ &\notag \hspace{-2.5cm}=\sum_{\vej=-\veq}^{\veq}t_{\vej}\,\, T_{n_1}\left(\begin{bmatrix} 0 & {\rm e}^{\ii j_1\theta_1}\\ {\rm e}^{-\ii j_1\theta_1} & 0\end{bmatrix}\right)\otimes \cdots\otimes T_{n_d}\left(\begin{bmatrix} 0 & {\rm e}^{\ii j_d\theta_d}\\ {\rm e}^{-\ii j_d\theta_d} & 0\end{bmatrix}\right)+R_{\ven}+E_{\ven}\\
& \hspace{-2.5cm}=\Tn\left(\begin{bmatrix} 0 & f_{\veq}\\ f_{\veq}^{\ast} & 0\end{bmatrix}\right)+R_{\ven}+E_{\ven}
\end{align}
with
\[\lim_{\ven\rightarrow\infty} \frac{{\rm rank}(R_{\ven})}{d_{\ven}}=0, \,\, \lim_{\ven\rightarrow\infty}\|E_{\ven}\|=0,\]
or equivalently 
\begin{equation*}
{\left\{\Tn\left(\begin{bmatrix} 0 & f_{\veq}\\ f_{\veq}^{\ast} & 0\end{bmatrix}\right)\right\}\sn}\xrightarrow{a.c.s.} \{\Pin\Un\Yn\Tn(f_{\veq})\Un\Pin^T\}\sn.
\end{equation*}
Thanks to \textbf{GLT2--4} and \cref{thm:zero_dist} the thesis is proven for $f=f_{\veq}$ a trigonometric polynomial.

Let us now switch to a generic $f\in L^{1}([-\pi,\pi]^d)$. It is well known that the set of $d$-variate polynomials is dense in $L^1([-\pi,\pi]^d)$. Therefore, there exists a sequence of polynomials $f_m:[-\pi,\pi]^d\rightarrow \mathbb{C}$ such that $f_m\rightarrow f\in L^{1}([-\pi,\pi]^d)$. By \cref{prop:acs_toeplitz} $\{\Tn(f_m)\}_n{\xrightarrow{a.c.s.}}\{\Tn(f)\}\sn$ i.e., for every $m$ there exists $n_m$ such that, for $n>n_m$,
\begin{equation*}
\Tn(f)=\Tn(f_{m})+R_{\ven,m}+E_{\ven,m}
\end{equation*}
with 
\begin{equation*}
{\rm rank}(R_{\ven,m})<c(m)d_{\ven}, \quad \|E_{\ven,m}\|\le\omega(m), \quad \lim_{m\rightarrow\infty}c(m),\omega(m)=0. 
\end{equation*}
Now, by \cref{eq:pol} we have
\begin{align*}
\Pin\Un\Yn\Tn(f)\Un\Pin^T&\\
&\hspace{-2.2cm}=\Pin\Un\Yn\Tn(f_{m})\Un\Pin^T  + \Pin\Un\Yn R_{\ven,m} \Un\Pin^T+ \Pin\Un\Yn E_{\ven,m} \Un\Pin^T\\
&\hspace{-2.2cm}=  \Tn\left(\begin{bmatrix} 0 & f_{m}\\ f_{m}^{\ast} & 0\end{bmatrix}\right)+\underbrace{R_{\ven}  + \Pin\Un\Yn R_{\ven,m} \Un\Pin^T}_{\tilde{R}}+ \underbrace{E_{\ven}+ \Pin\Un\Yn E_{\ven,m} \Un\Pin^T}_{\tilde{E}}
\end{align*}
with  
\begin{align*}
{\rm rank}({\tilde{R}})<\tilde{c}(m)d_{\ven}, \quad \|{\tilde{E}}\|\le\tilde{\omega}(m),\quad \lim_{m\rightarrow\infty}\tilde{c}(m),\tilde{\omega}(m)=0. 
\end{align*}
Then, 
\begin{equation*}
\left\{\Tn\left(\begin{bmatrix} 0 & f_{m}\\ f_{m}^{\ast} & 0\end{bmatrix}\right)\right\}_n{\xrightarrow{a.c.s.}}\{\Pin\Un\Yn\Tn(f)\Un\Pin^T\}\sn.
\end{equation*}
This together with $\left\{\begin{bmatrix} 0 & f_{m}\\ f_{m}^{\ast} & 0\end{bmatrix}\right\}_m\rightarrow \begin{bmatrix} 0 & f\\ f^{\ast} & 0\end{bmatrix}$ with $f\in L^1([-\pi,\pi]^d)$, and \textbf{GLT3} and \textbf{GLT5}, concludes the proof.
\end{proof} 

\begin{remark}\label{rem:odd}
Assume that $n_k=2m_k+1$ with $m_k\in\mathbb{N}$. Then, $U_{n_k}Y_{n_k}T_{n_k}({\rm e}^{\ii j_k\theta_k})U_{n_k}$ can be embedded into the $(2m_k+2)\times(2m_k+2)$ matrix
\begin{align*}
A_{n_k+1}&\\
& =\begin{bmatrix}H_{m_k+1}({\rm e}^{\ii j_k\theta_k})& T_{m_k+1}({\rm e}^{\ii j_k\theta_k}) \\ T_{m_k+1}({\rm e}^{-\ii j_k\theta_k})& H_{m_k+1}({\rm e}^{-\ii j_k\theta_k})\end{bmatrix}\\
&=\Pi_{n_k+1}^TT_{2m_k+2}\left(\begin{bmatrix} 0 & {\rm e}^{\ii j_k\theta_k}\\ {\rm e}^{-\ii j_k\theta_k} & 0\end{bmatrix}\right)\Pi_{n_k+1}\\
&\hspace{0.5cm}+\Pi_{n_k+1}^TH_{2m_k+2}\left(\begin{bmatrix}  {\rm e}^{\ii j_k\theta_k} & 0\\ 0 & {\rm e}^{-\ii j_k\theta_k}  \end{bmatrix}\right)\Pi_{n_k+1}\\
&=\Pi_{n_k+1}^TT_{2m_k+2}\left(\begin{bmatrix} 0 & {\rm e}^{\ii j_k\theta_k}\\ {\rm e}^{-\ii j_k\theta_k} & 0\end{bmatrix}\right)\Pi_{n_k+1}\\
&\hspace{0.5cm}+ R_{2m_k+2} +E_{2m_k+2}
\end{align*}
where $H(\cdot)$ is the (block) Hankel matrix generated by the function specified by the brackets and the last equality follows from \cref{lem:hankel-zero-symbol} combined with \cref{thm:zero_dist}.
Specifically, 
\begin{equation*}
U_{n_k}Y_{n_k}T_{n_k}({\rm e}^{\ii j_k\theta_k})U_{n_k}=PA_{n_k+1}P^T,
\end{equation*}
with 
\begin{equation*}
P=\begin{bmatrix}I_{m+1}& {\bf 0}&  O_{(m+1) \times m} \\ O_{m\times (m+1)} & {\bf 0} & I_m\end{bmatrix} \text{ and } {\bf 0}=(0, \ldots, 0)^T.
\end{equation*}
On this basis, using the same line of proof as for \cref{lem:Y_GLT} shows that the matrix $\Un\Yn\Tn(f)\Un$ is a principal submatrix of a matrix that, after a proper permutation, gives rise to a GLT sequence whose symbol is $g$. 
\end{remark}

\begin{remark}\label{rem:perm_T}
Assume that $\ven=(n_1,\ldots,n_d)$ with $n_k=2m_k$, $m_k \in \mathbb{N}$. Then, the  line of proof used for \cref{lem:Y_GLT} shows that
\begin{equation*}\label{eq:perm_T}
\{\Pin\Tn(f)\Pin^T\}\sim_{\rm GLT}\begin{bmatrix} f & 0\\ 0& f\end{bmatrix}
\end{equation*}
and
\begin{equation}\label{eq:perm_UTU}
\{\Pin\Un\Tn(f)\Un\Pin^T\}\sim_{\rm GLT}\begin{bmatrix}f^{\ast} & 0\\ 0 & f\end{bmatrix}.
\end{equation}
\end{remark}
\begin{remark}\label{rem:distrY}
Assume that $\ven=(n_1,\ldots,n_d)$ with $n_k=2m_k$, $m_k \in \mathbb{N}$. Then,
\begin{align*}
\Pin\Un\Yn\Un\Pin^T&=\Pi_{n_1}U_{n_1}Y_{n_1}U_{n_1}\Pi_{n_1}^T\otimes \cdots \otimes \Pi_{n_d}U_{n_d}Y_{n_d}U_{n_d}\Pi^T_{n_d}\\
&=T_{n_1}\left(\begin{bmatrix} 0 & 1\\ 1 & 0\end{bmatrix}\right)\otimes \cdots \otimes T_{n_d}\left(\begin{bmatrix} 0 & 1\\ 1 & 0\end{bmatrix}\right)\\
&=\Tn\left(\begin{bmatrix} 0 & 1\\ 1 & 0\end{bmatrix}\right).
\end{align*}
Hence, using \cref{eq:perm_UTU} we arrive at the same result as in \cref{lem:Y_GLT}, i.e.,
\begin{align*}
\{\Pin\Un\Yn\Tn(f)\Un\Pin^T\}\\
&\hspace{-1.2cm}=\{\Pin\Un\Yn\Un\Pin^T\Pin\Un\Tn(f)\Un\Pin^T\}\sim_{\rm GLT}\begin{bmatrix} 0 & 1\\ 1 & 0\end{bmatrix}\begin{bmatrix}f^{\ast} & 0\\ 0 & f\end{bmatrix}=g.
\end{align*}
\end{remark}
We can now state the main theorem of this section, which describes the spectral distribution of $\{\Yn\Tn(f)\}\sn$. 
\begin{theorem}
\label{thm:main}
Let $\{\Tn(f)\}\sn$, $\Tn(f)\in\mathbb{R}^{\npr\times \npr}$ be the multilevel Toeplitz sequence associated with $f \in L^1([-\pi,\pi]^d)$, \
where $\ven=(n_1,\ldots,n_d)$ and $\npr = n_1\dotsm n_\dime$.
Let $\{\Yn\Tn(f)\}\sn$ be the corresponding sequence of flipped Toeplitz matrices. Then,
\begin{align}\label{eq:distrg}
\{\Yn\Tn(f)\}\sn\sim_\lambda (g,[-\pi,\pi]^d), \text{\,\,\, with \,\,\,} 
g=\begin{bmatrix} 0 & f\\ f^\ast & 0\end{bmatrix}.
\end{align}
\end{theorem}

\begin{proof}
In the case that $n_k = 2m_k$, $m_k \in \mathbb{N}$ for each $k$, we see from \cref{lem:Y_GLT} 
that $\{\Pin\Un\Yn\Tn(f)\Un\Pin^T\} \sim_{\rm GLT} g$. Hence, recalling that $\Pin\Un\Yn\Tn(f)\Un\Pin^T$ is real symmetric, by \textbf{GLT1},
$\{\Yn\Tn(f)\} \sim_{\lambda} g$.
In all other cases, by recalling \cref{rem:odd} and using \cref{th:extradimensional} we find that the thesis follows as well.
\end{proof}

We end this section by providing the spectral distribution of a preconditioned sequence of flipped multilevel Toeplitz matrices.

\begin{theorem}\label{thm:precflip}
	Let $\{\Tn(f)\}\sn$,  $\Tn(f)\in\mathbb{R}^{\npr\times \npr}$ with $n_k=2m_k$, $m_k\in\mathbb{N}$ be the multilevel Toeplitz sequence associated with $f\in L^{1}([-\pi,\pi]^d)$, let $\{\Yn\Tn(f)\}\p$ be the corresponding sequence of flipped multilevel Toeplitz matrices, and let $\{\Prec\p\}\p$ be a sequence of Hermitian positive definite matrices such that $\{\Prec\p\}\p\sim_{\rm GLT}h$, and $\{\Pin\Un\Prec\p\Un\Pin^T\}\p\sim_{\rm GLT}h\cdot I_2$ with $h:[-\pi,\pi]^d\rightarrow\mathbb{C}$ and $h\ne0$ a.e. Then,
	\begin{equation}\label{eq:precond}
	\{\Prec\p^{-1}\Yn\Tn(f)\}\p\sim_{\lambda}(h^{-1}g,[-\pi,\pi]^d).
	\end{equation}
\end{theorem}
\proof 
The thesis follows from the combination of \cref{thm:prec} and \cref{lem:Y_GLT} by noticing that
\begin{equation*}
(\Pin\Un\Prec\p\Un\Pin^T)^{-1}\Un\Pin\Yn\Tn(f)\Un\Pin^T=\Pin \Un\Prec\p^{-1}\Yn\Tn(f)\Un\Pin^T
\end{equation*}
and by recalling that $\Yn\Tn(f)$ is real symmetric and that $\Pin\Un$ is orthogonal.
\endproof
\begin{sloppypar}
Note that, thanks to \cref{rem:perm_T}, the hypotheses of \cref{thm:precflip} are satisfied in the case where $\Prec\sn=\Tn(h)$, with $h\ge0$ and $h\ne0$ a.e. Moreover, it easy to see that if we take the following circulant preconditioner $$\mathcal{P}\sn=C_{\ven}=|C_{n_d}| \otimes \cdots \otimes I_{n_1}+ \cdots + I_{n_d} \otimes \cdots \otimes |C_{n_1}|$$ with $|C_{n_\ell}|=(C_{n_\ell}^T C_{n_\ell})^{\frac12}$ where $C_{n_\ell}$ is the optimal preconditioner for $T_{n_\ell}=T_{n_\ell}(f_{\ell})$, the condition $\{\Pin\Un C_{\ven}\Un\Pin^T\}\p\sim_{\rm GLT}h\cdot I_2$ holds as well. This is because both $\{|C_{n_{\ell}}|\}_{n_\ell}$ and $\{T(|f_{\ell}|)\}_{n_\ell}$ are GLTs with symbol $|f_\ell|$ and then  $|C_{n_{\ell}}|=T(|f_{\ell}|)+R_{n_{\ell}}+E_{n_{\ell}}$ which allows us to apply the same reasoning as in \cref{rem:perm_T} to prove the desired relation.
\end{sloppypar}

\section{Numerical results} \label{sec:num}
In this section we illustrate the theoretical results from \cref{sec:main}, that is, we check the validity of  \cref{thm:main,thm:precflip}. We start by defining the following equispaced grid on $[0,\pi]^d$:
\begin{eqnarray*}
\Gamma=\left\{(\theta^{(k_1)}_1,\ldots,\theta^{(k_d)}_{d}):=\left(\frac{\pi k_1}{\left\lfloor\frac{n_1}{2}\right\rfloor-1},\frac{\pi k_2}{n_2-1},\ldots,\frac{\pi k_d}{n_d-1}\right),
\begin{array}{rl}
k_1=&0,\ldots,\left\lfloor\frac{n_1}{2}\right\rfloor-1,\\
k_j=&0, \dots, n_j-1,\\
j=&2,\ldots,d
\end{array}
\right\}.
\end{eqnarray*}
Then, we denote by $\Lambda_1$ and $\Lambda_2$ the set of all evaluations of $\lambda_1(g)$, $\lambda_2(g)$ (resp.\ $\lambda_1(h^{-1}g)$, $\lambda_2(h^{-1}g)$) on $\Gamma$, and by $\Lambda$ the union $\Lambda_1\cup\Lambda_2$ ordered in an ascending way. In the following examples we numerically check relation \cref{eq:distrg} (resp.\ \cref{eq:precond}) by comparing the eigenvalues of $\Yn\Tn(f)$ (resp. $\Prec_n^{-1}\Yn\Tn(f)$) with the values collected in $\Lambda$. Note that it suffices to consider only $[0,\pi]^d$ in place of $[-\pi,\pi]^d$ because the eigenvalue functions of the considered symbols are even. 

In the two-dimensional  \cref{ex:1,ex:2} we also compare the eigenvalues of $\Yn\Tn(f)$ directly with the spectrum of $g$ over the whole domain $[-\pi,\pi]^2$. Precisely, we define the following grid on $[-\pi,\pi]^2$ 
\begin{eqnarray*}
\Delta=\left\{(\theta^{(\ell)}_1,\theta^{(\kappa)}_{2}):=\left(-\pi+\frac{2\pi \ell}{n_1-1},-\pi+\frac{2\pi \kappa}{n_2-1}\right),
\begin{array}{rl}
\ell=&0, \dots, n_1-1,\\
\kappa=&0, \dots, n_2-1
\end{array}
\right\}
\end{eqnarray*}
and again we denote by $\Lambda_1$ and $\Lambda_2$ the sets of all evaluations of $\lambda_1(g)$, $\lambda_2(g)$ on $\Delta$, and by $\Lambda$ the union $\Lambda_1\cup\Lambda_2$ ordered in an ascending way. Therefore, we employ the following matching algorithm: for a fixed eigenvalue $\lambda$ of $\Yn\Tn(f)$
\begin{enumerate}
	\item we find $\tilde{\eta}\in \Lambda$ such that $\|\lambda-\tilde{\eta}\|=\min_{\eta\in\Lambda}\|\lambda-\eta\|$, and
	\item we associate $\lambda$ to the couple in $\Delta$ that corresponds to $\tilde{\eta}$.
\end{enumerate}

\begin{example} \label{ex:1}
The first example we consider is the $2$-level banded Toeplitz matrix generated by $f(\theta_1,\theta_2) = 4+\ee^{\ii\theta_1} + \ee^{\ii\theta_2}$. 
We see from \cref{fig:ex1} that the uniform sampling of eigenvalue functions of $g$ collected in $\Lambda$ accurately describes the eigenvalues of $\Yn\Tn(f)$, even for very small matrices. Moreover, as shown in \cref{fig:ex1_surface} (obtained using the aforementioned matching algorithm), the eigenvalues of $\Yn\Tn(f)$  accurately mimic the shape of the eigenvalue functions of $g$ when $n_1=20, n_2=40$.

\begin{figure}[htbp]
	\centering
	\begin{subfigure}[b]{0.48\textwidth}
		\includegraphics[width = \textwidth,clip=true,trim = 0cm 7cm 0cm 7cm]{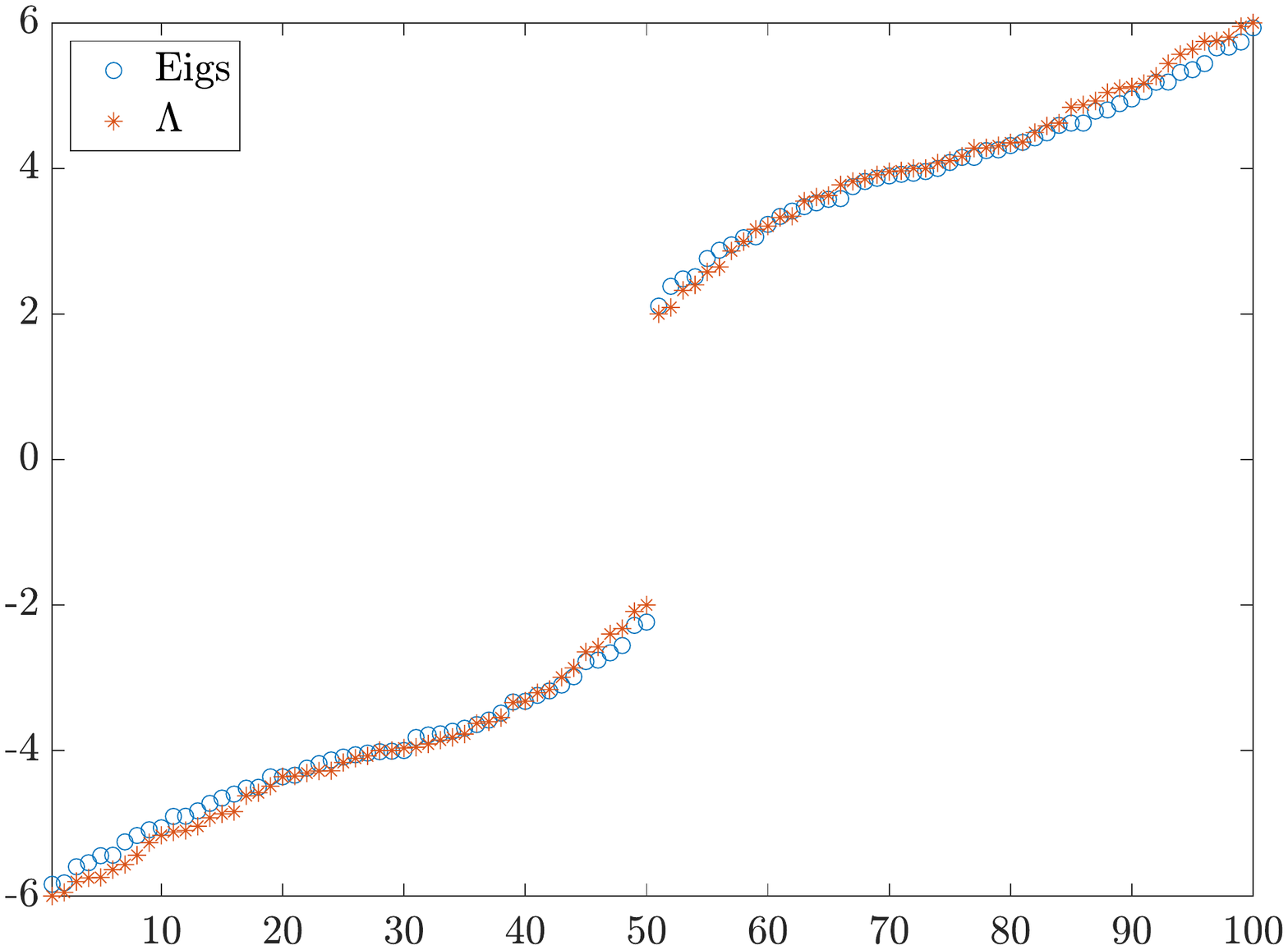}
	\end{subfigure}
	\begin{subfigure}[b]{0.48\textwidth}
		\includegraphics[width = \textwidth,clip=true,trim = 0cm 7cm 0cm 7cm]{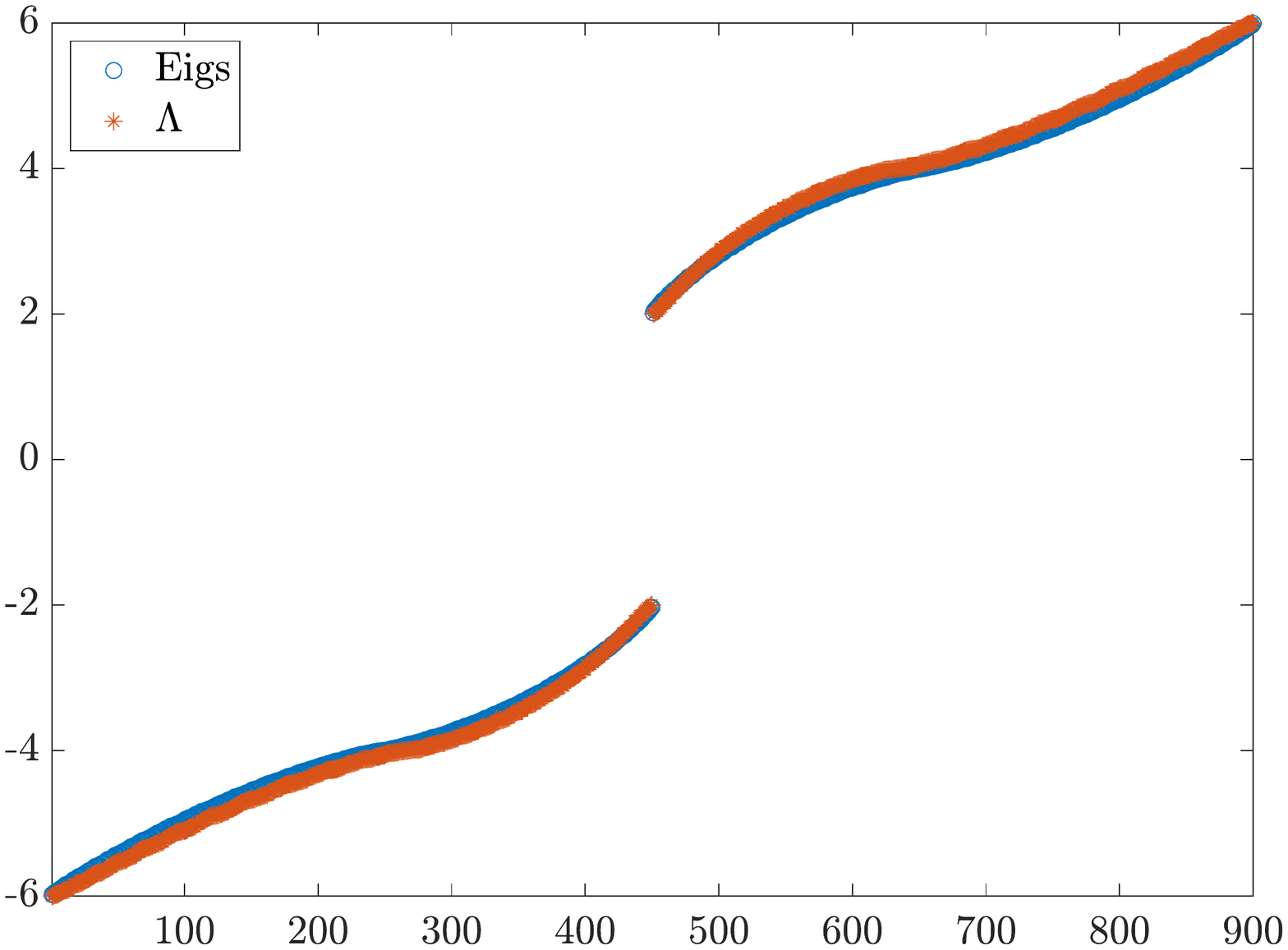}
	\end{subfigure}
	\caption{Comparison of the eigenvalues of $Y_{\ven}T_{\ven}(f)$ ({\color{blue} $\circ$}) with $\Lambda$ collecting the uniform samples of the eigenvalue functions of $g$ for \cref{ex:1} ({\color{red}$\ast$}) when $n_1 = n_2 = 10$ (left) and $n_1 = n_2 = 30$ (right).} 
	\label{fig:ex1}
\end{figure}

\begin{figure}[htbp]
	\centering
	\begin{subfigure}[b]{0.48\textwidth}
		\includegraphics[width = \textwidth,clip=true,trim = 0cm 7cm 0cm 7cm]{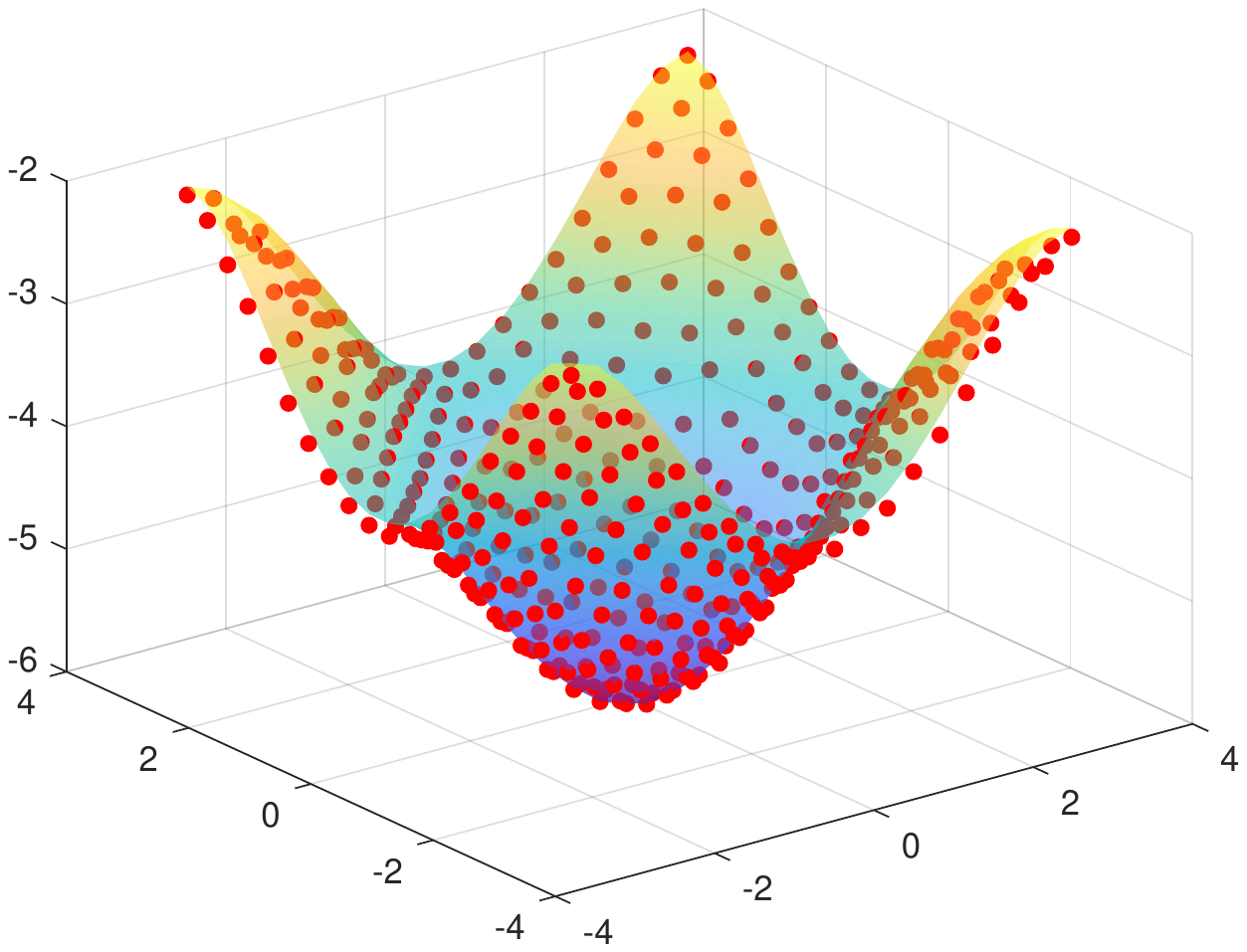}
						\subcaption{$\lambda_1(g)=-|f|$}
	\end{subfigure}
	\begin{subfigure}[b]{0.48\textwidth}
		\includegraphics[width = \textwidth,clip=true,trim = 0cm 7cm 0cm 7cm]{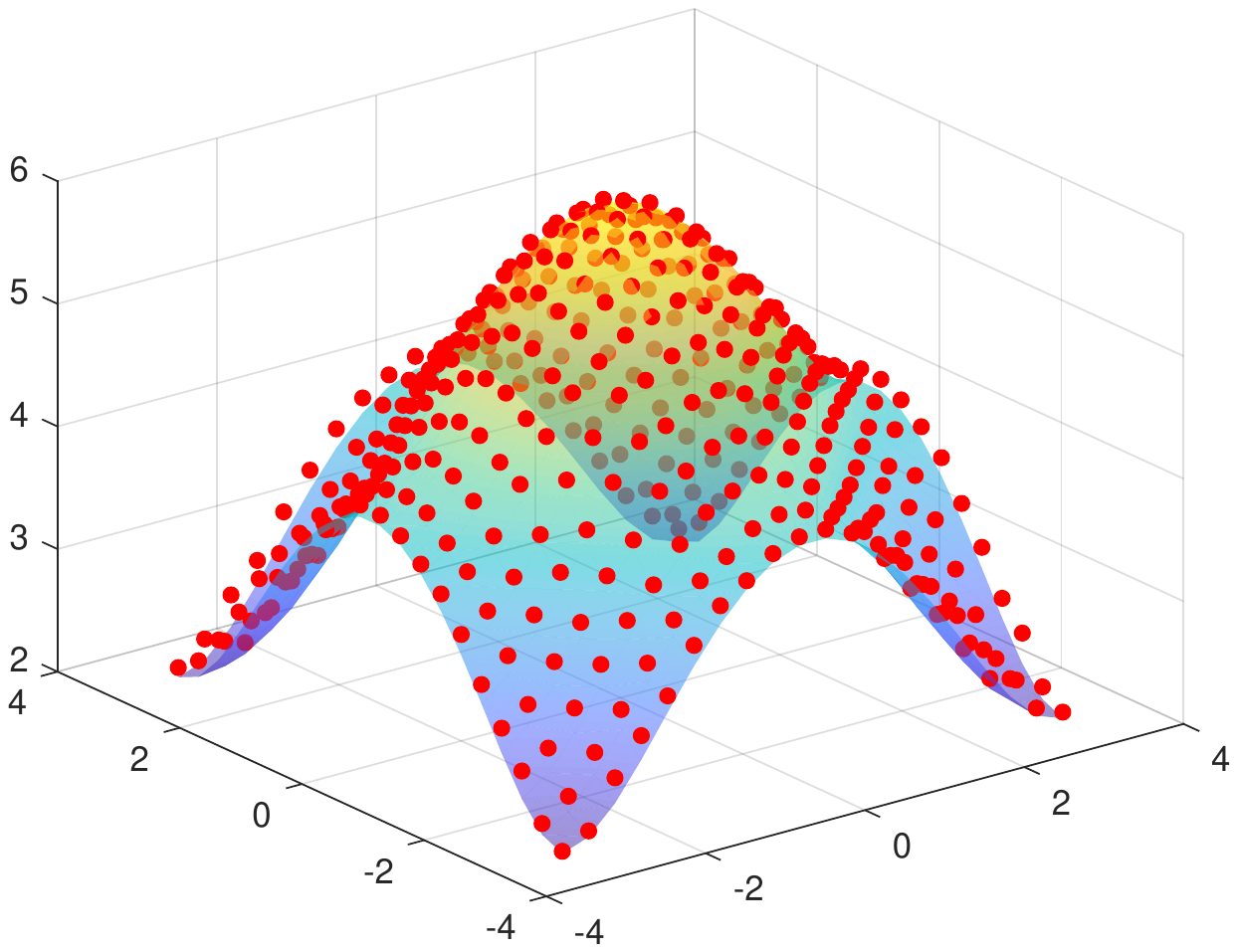}
						\subcaption{$\lambda_2(g)=|f|$}
	\end{subfigure}
	\caption{Eigenvalues of $Y_{\ven}T_{\ven}(f)$ (red dots) and the spectrum of $g$ (colored surfaces) for \cref{ex:1} when $n_1 = 20$ and $n_2 = 40$.} 
	\label{fig:ex1_surface}
\end{figure}

\end{example}

\begin{example} \label{ex:2}
In this example we consider the dense $2$-level Toeplitz matrix obtained by discretizing a certain time-dependent initial-boundary fractional diffusion problem by means of a second-order finite difference approximation that combines the Crank-Nicolson scheme and the so-called weighted and shifted Gr\"unwald formula (see \cite{Tian2015}). Precisely, we start from
\begin{align*}
\begin{cases}
\frac{\partial u(x,y,t)}{\partial
	t}=\frac{\partial^{\alpha}
	u(x,y,t)}{\partial_{+}x^{\alpha}}+\frac{\partial^{\beta}
	u(x,y,t)}{\partial_{+}y^{\beta}}+v(x,y,t), \qquad \qquad &(x,y,t)\in\Omega\times(0,1],\\
u(x,y,t)=0,  &(x,y,t)\in{\mathbb{R}^2\backslash\Omega}\times[0,1],\\
u(x,y,0)=u_0(x,y), &(x,y)\in\bar{\Omega},
\end{cases}
\end{align*}
where $\Omega=(0,1)\times(0,1)$, $\alpha,\beta\in(1,2)$, and $\frac{\partial^{\alpha}
	u(x,y,t)}{\partial_{+}x^{\alpha}}$, $\frac{\partial^{\alpha}
	u(x,y,t)}{\partial_{+}y^{\beta}}$ are
fractional derivatives defined in Riemann-Liouville form (see again \cite{DKMT20}). Then, for fixed $n_1,n_2,M\in\mathbb{N}$, we take the following equispaced partition of $\Omega\times [0,1]$
\begin{eqnarray*}
	x_i=ih_x,&\quad h_x=\frac{1}{n_1+1},&\quad i=0,1,\dots,n_1,\\
	y_j=jh_y,&\quad h_y=\frac{1}{n_2+1},&\quad j=0,1,\dots,n_2,\\
	t^{(m)}=m\Delta t,& \; \Delta t=\frac{1}{M},\quad &\quad m=0,1,\dots,M,
\end{eqnarray*}
and we arrive at a linear system whose coefficient matrix is the $2$-level Toeplitz matrix

\begin{align*}
\mathbf{\mathcal{M}}^{(\alpha,\beta)}_{\ven} &=\frac{2h_x^\alpha}{\Delta t}I_{n_1n_2}+I_{n_2}\otimes T_{n_1}(f_\alpha)+\frac{h_x^\alpha}{h_x^\beta}T_{n_2}(f_\beta)\otimes I_{n_1},
\end{align*}
with $\ven=(n_1,n_2)$,
\begin{equation*}\label{eq:fgamma}
f_{\gamma}(\theta)=-\sum_{k=-1}^\infty w_{k+1}^{(\beta)}{\rm e}^{\mathbf{i} k\theta}=-\bigg[\frac{2-\gamma(1-{\rm
		e}^{-\mathbf{i}\theta})}{2}\bigg]\left(1+{\rm
	e}^{\mathbf{i}(\theta+\pi)}\right)^{\gamma}, 
\end{equation*}
$\gamma\in\{\alpha,\beta\}$, $\theta\in\{\theta_1,\theta_2\}$, and $w_k^{(\beta)}$ defined as in \cite{Tian2015}. Both $T_{n_1}(f_\alpha)$ and $T_{n_2}(f_\beta)$ are lower Hessenberg, and so $\mathbf{\mathcal{M}}^{(\alpha,\beta)}_{\ven}$ is non-symmetric.

It has been shown in \cite{Moghaderi2017} that, whenever $\frac{h_x^\alpha}{h_x^\beta}=O(1)$ and $\frac{2h_x^\alpha}{\Delta t}=o(1)$, it holds that $$\{\mathbf{\mathcal{M}}^{(\alpha,\beta)}_{\ven}\}_{\ven}\sim_{\lambda} f_{\alpha,\beta}:=f_\alpha(\theta_1)+\frac{h_x^\alpha}{h_x^\beta}f_\beta(\theta_2),$$
i.e., $\mathbf{\mathcal{M}}^{(\alpha,\beta)}_{\ven}=T_{\ven}(f_{\alpha,\beta})$.

In the following tests we fix $\alpha=1.8$ and $\beta=1.6$. \cref{fig:ex2}(a) shows that when $M=n_1=30$ and $n_2=35$ the eigenvalues of the flipped Toeplitz matrix $\Yn\Tn(f)$ are well described by the sampling of the eigenvalue functions of $g$ given in $\Lambda$. Similar results can be inferred from \cref{fig:ex2_surface} when comparing the eigenvalues of $\Yn\Tn(f)$ directly when the spectrum of $g$ with $n_1=20$, $n_2=40$.

For this example we also show how the results in \cref{sec:main} can be used to describe the convergence rate of preconditioned MINRES, which depends heavily on the spectral properties of the coefficient matrix (see, e.g., \cite[Chapters 2 \& 4]{ESW14}). With this aim we focus on the solution of the following linear system
\begin{equation*}\label{eq:linsys}
T_{\ven}(f_{\alpha,\beta})u_{\ven}=b_{\ven}, 
\end{equation*}
with $b_{\ven}=2h_x^\alpha\ve{1}$, and we define the following preconditioners for $\Y_{\ven}T_{\ven}(f_{\alpha,\beta})$:
\begin{itemize}
	\item $\Prec\sn=T_{\ven}(f_R)$, with $f_R=\frac{f_{\alpha,\beta}+f^\ast_{\alpha,\beta}}{2}$. Of course, in this case the symbol of the preconditioning matrix sequence is $h=f_R$;
	\item $\Prec\sn=P^{2,2}_{\ven}$, obtained from $\mathbf{\mathcal{M}}^{(\alpha,\beta)}_{\ven}$ replacing $T_{n_1}(f_\alpha)$, $T_{n_2}(f_\beta)$, with $T_{n_1}(2-2\cos\theta_1)$, $T_{n_2}(2-2\cos\theta_2)$, respectively (see \cite{Moghaderi2017} for more details). In this case, the symbol of the preconditioning matrix sequence is $h=2-2\cos\theta_1+\frac{h_x^\alpha}{h_x^\beta}(2-2\cos\theta_2)$;
	\item 
	\begin{sloppypar}
	$\Prec\sn=P^{2,\beta}_{\ven}$, obtained from $\mathbf{\mathcal{M}}^{(\alpha,\beta)}_{\ven}$ replacing $T_{n_1}(f_\alpha)$ with $T_{n_1}(2-2\cos(\theta_1))$, and $T_{n_2}(f_\beta)$ with the real part of its tetra-diagonal band truncation $T_{n_2}(p_\beta(\theta_2))$, where
\begin{equation*}
p_\beta(\theta_2)=-\sum_{k=-1}^2w_{k+1}^{(\beta)}{\rm e}^{\i k\theta_2}.
\end{equation*}
\end{sloppypar}
In this case, the symbol of the preconditioning matrix sequence is $h=2-2\cos\theta_1+\frac{h_x^\alpha}{h_x^\beta}\frac{p_{\beta}(\theta_2)+p^\ast_{\beta}(\theta_2)}{2}$.
\end{itemize}
All of the aforementioned preconditioners are symmetric positive definite matrices that satisfy the conditions of \cref{thm:precflip} when $\ven=(n_1,n_2)$ has even components. \cref{fig:ex2}(b)--(d) show that the eigenvalues of $\Prec\sn^{-1}\Yn\Tn(f)$ are well described by the sampling of the eigenvalue functions of $h^{-1}g$ contained in $\Lambda$ even though not all components of $\ven$ are even, as required by \cref{thm:precflip} (here $n_1=30$ and $n_2=35$). Moreover, in all the given cases the eigenvalues of the preconditioned matrices lie close to $1$ and $-1$. This is particularly evident for $\Prec\sn=\Tn(f_R)$. Note that, when $\Prec\sn=P_{\ven}^{2,\beta}$, the eigenvalue functions of $h^{-1}g$ assume values around zero (while the eigenvalues of $\Prec\sn^{-1}\Yn\Tn(f)$ do not); this is because $p_\beta(\theta_2)$ does not have a zero at $\theta_2=0$ but $f_\beta(\theta_2)$ does.

From \cref{fig:ex2}(b)--(d) and since $\lambda_1(h^{-1}g),\lambda_2(h^{-1}g)$ are clustered at $\pm1$, we expect that preconditioned MINRES applied to the flipped version of \cref{eq:linsys} with preconditioners $\Tn(f_R)$, $P_{\ven}^{2,2}$ or $P_{\ven}^{2,\beta}$  will converge at a fast rate. In \cref{tab:ex2} the iterations of preconditioned MINRES are stopped when the residual norm is reduced by eight orders of magnitude, i.e, when $\|r_k\|_2/\|r_0\|_2 < 10^{-8}$. We see from these results that for all three preconditioners convergence is rapid, with $\Tn(f_R)$ resulting in the lowest iteration counts. Neither $P_{\ven}^{2,2}$ nor $P_{\ven}^{2,\beta}$  is optimal and this is in line with the spectral analysis performed in \cite{Moghaderi2017,DKMT20}. On the other hand, both are block banded with banded block matrices, and so are computationally affordable unlike the dense preconditioner $\Tn(f_R)$.

\begin{figure}[htbp]
	\centering
	\begin{subfigure}[b]{0.48\textwidth}
		\includegraphics[width = \textwidth,clip=true,trim = 0cm 7cm 0cm 7cm]{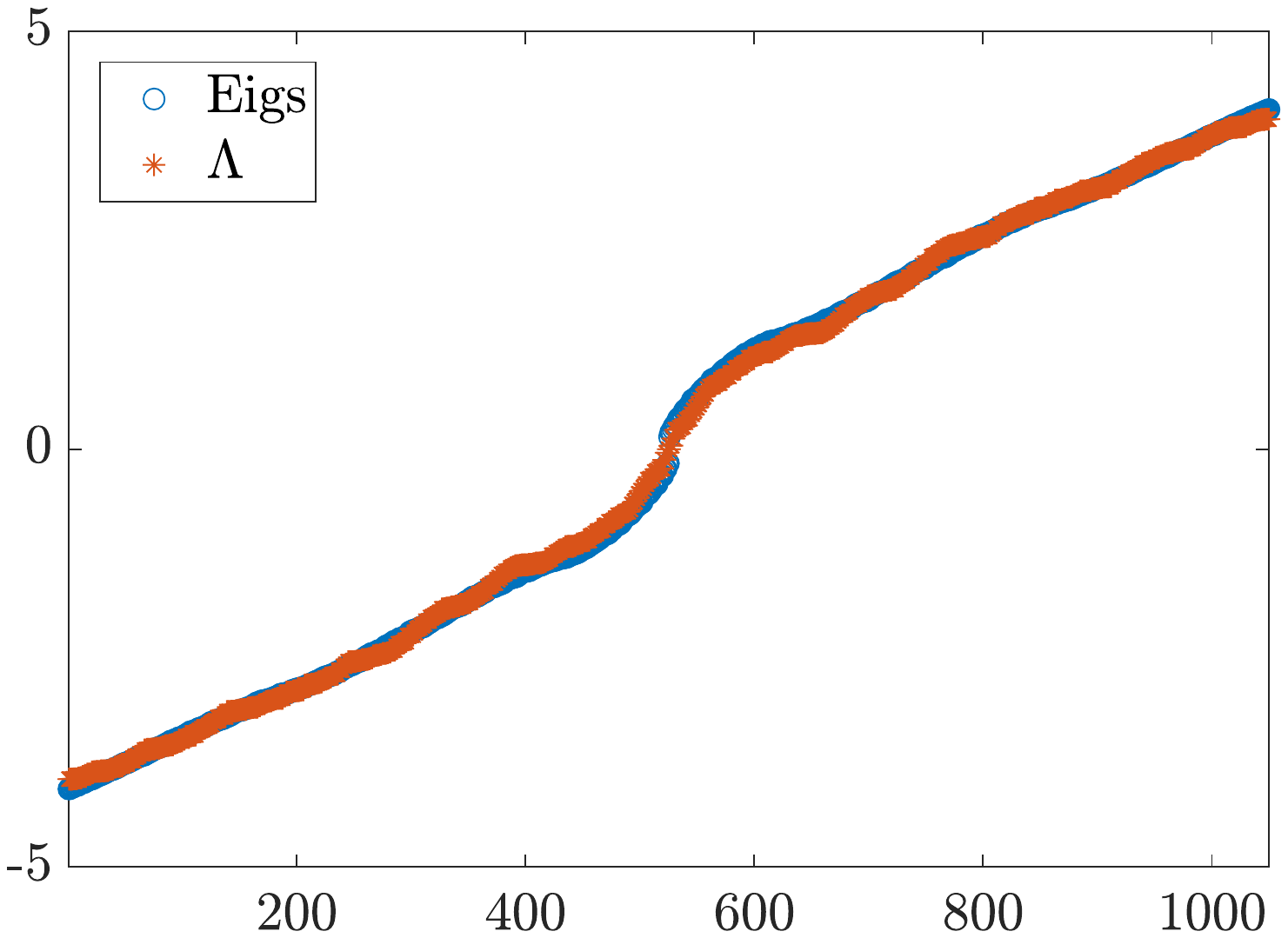}
		\subcaption{Unpreconditioned}
	\end{subfigure}
	\begin{subfigure}[b]{0.48\textwidth}
		\includegraphics[width = \textwidth,clip=true,trim = 0cm 7cm 0cm 7cm]{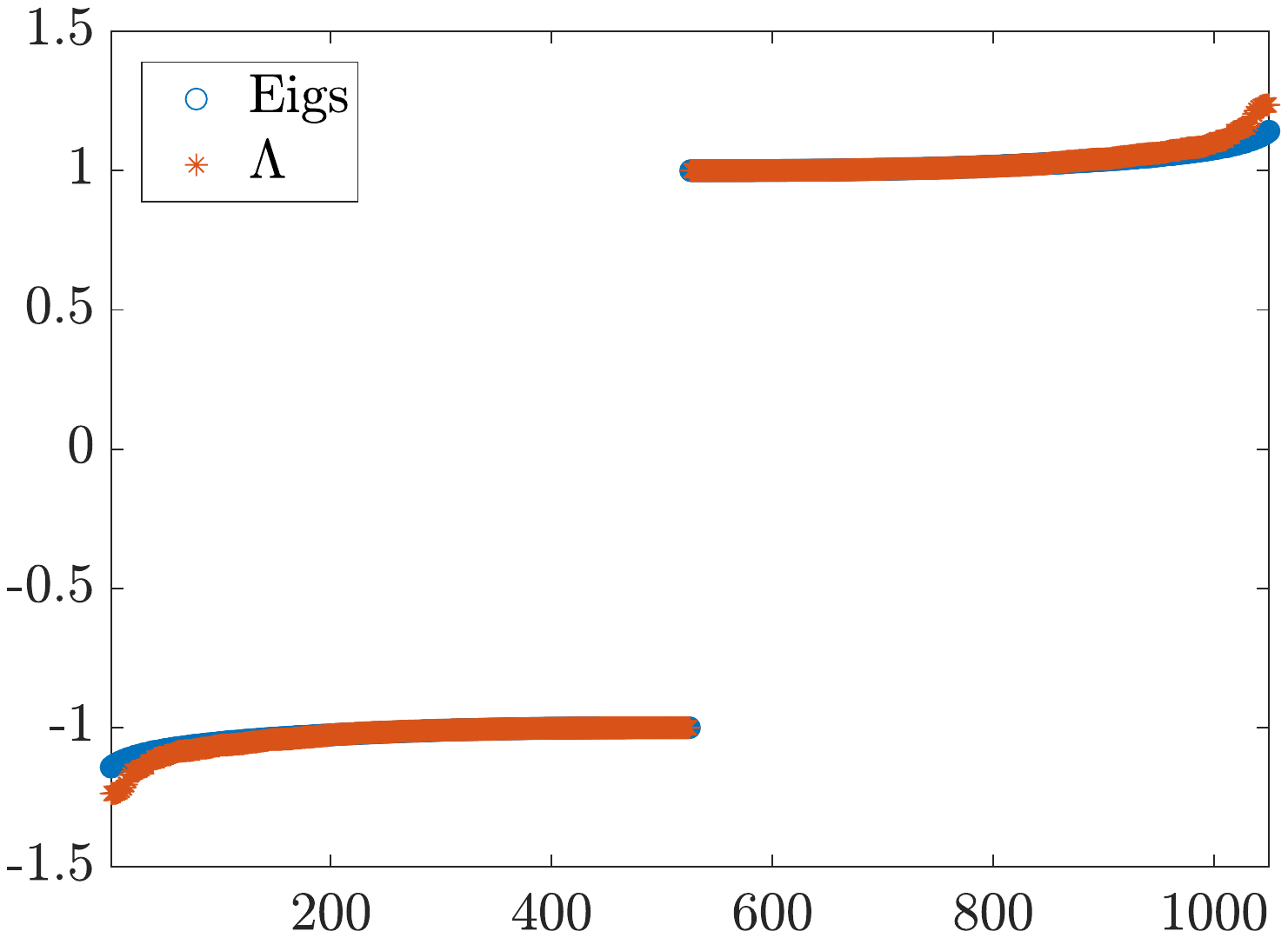}
		\subcaption{$\Prec\sn = T_{\ven}(f_R)$}
	\end{subfigure}
	
	\begin{subfigure}[b]{0.48\textwidth}
		\includegraphics[width = \textwidth,clip=true,trim = 0cm 7cm 0cm 7cm]{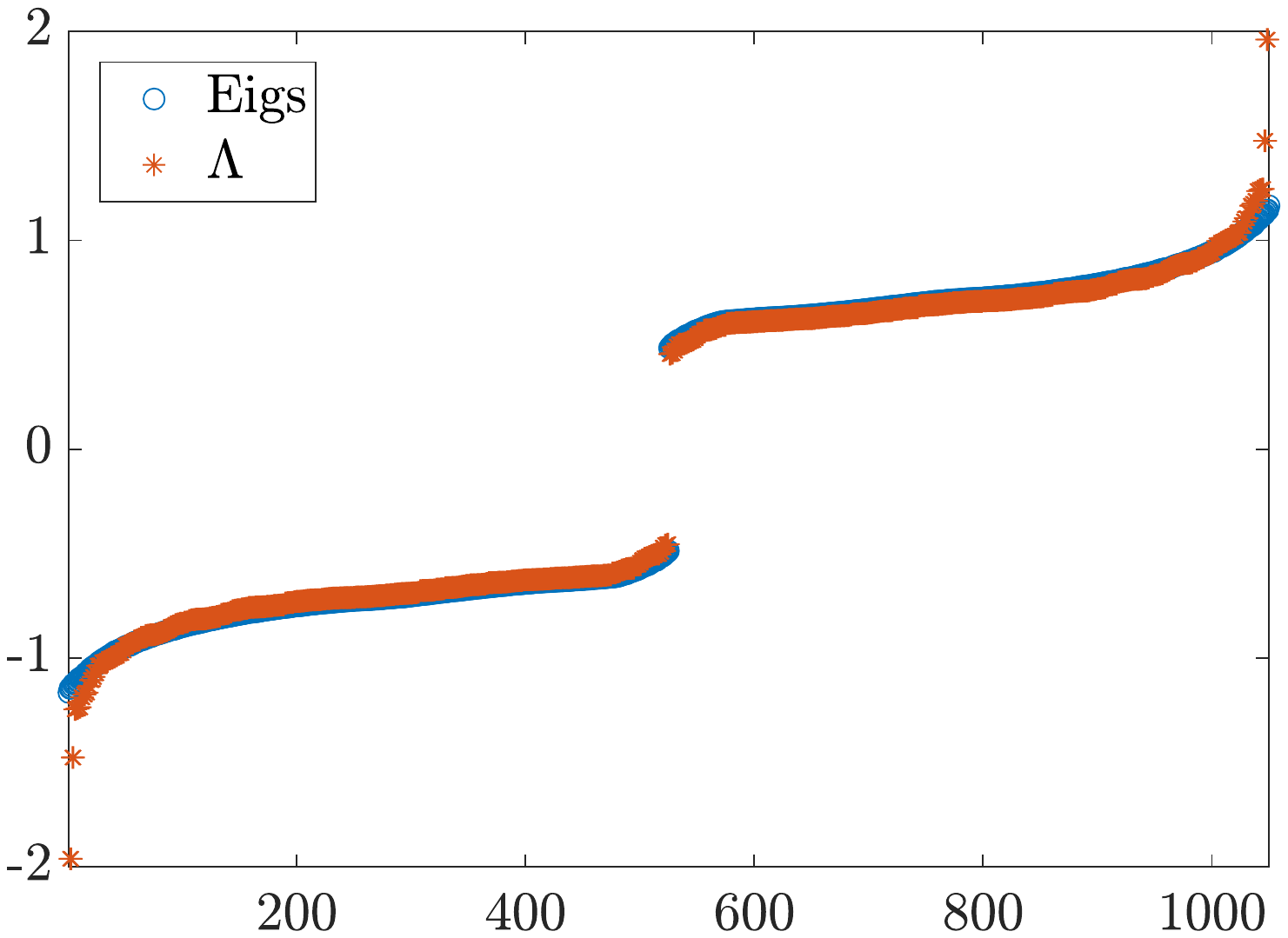}
		\subcaption{$\Prec\sn= P^{2,2}_{\ven}$}
	\end{subfigure}
	\begin{subfigure}[b]{0.48\textwidth}
		\includegraphics[width = \textwidth,clip=true,trim = 0cm 7cm 0cm 7cm]{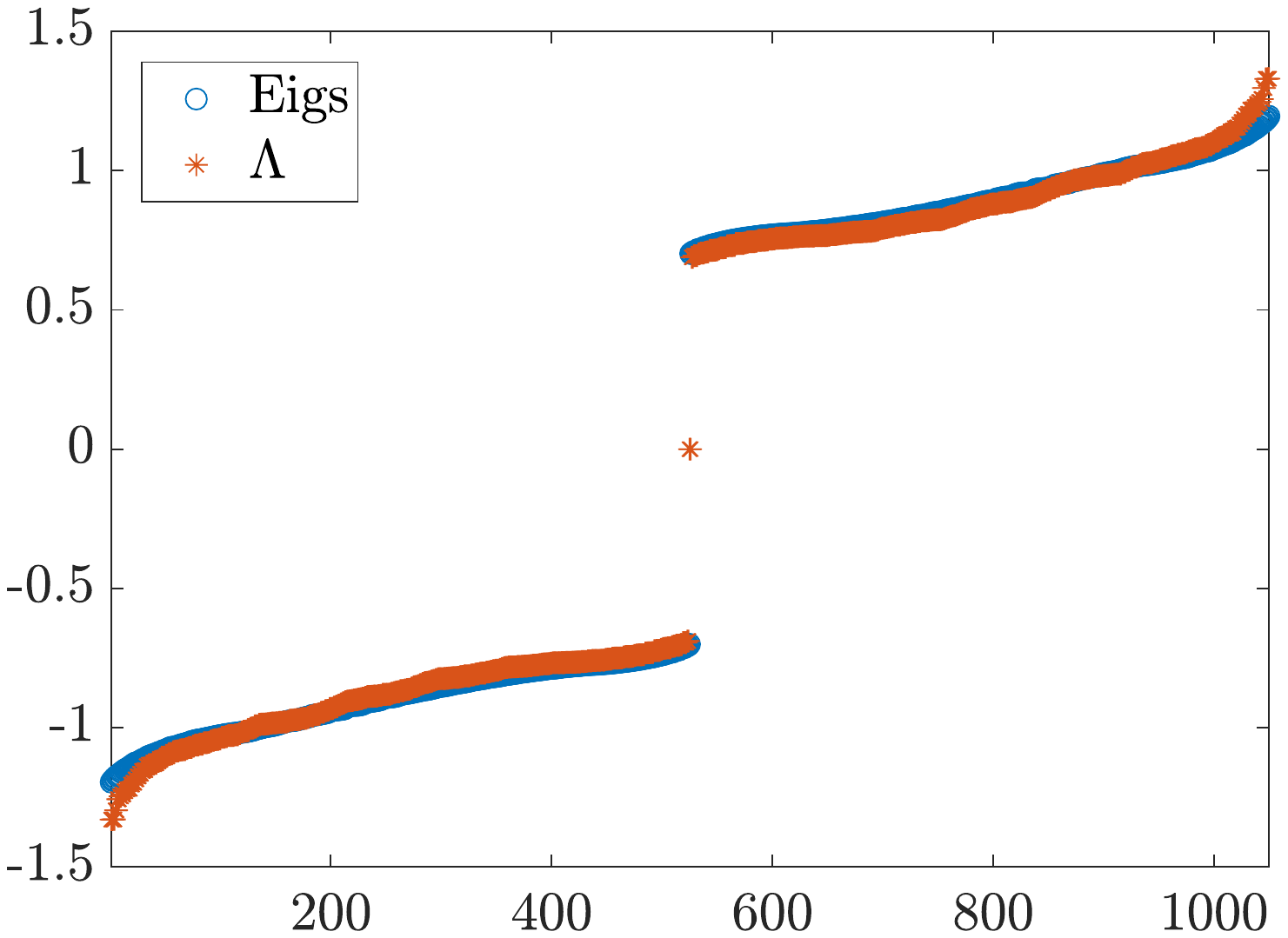}
		\subcaption{$\Prec\sn = P^{2,\beta}_{\ven}$}
	\end{subfigure}
	\caption{Comparison of the eigenvalues of $Y_{\ven}T_{\ven}(f_{\alpha,\beta})$ or $\Prec\sn^{-1}Y_{\ven}T_{\ven}(f_{\alpha,\beta})$ ({\color{blue} $\circ$}) with $\Lambda$ collecting the uniform samples of the eigenvalue functions of $g$ or $h^{-1}g$ for \cref{ex:2} ({\color{red}$\ast$}) when $M=n_1=30$, $n_2 = 35$.} 
	\label{fig:ex2}
\end{figure}

\begin{figure}[htbp]
	\centering
	\begin{subfigure}[b]{0.48\textwidth}
		\includegraphics[width = \textwidth,clip=true,trim = 0cm 7cm 0cm 7cm]{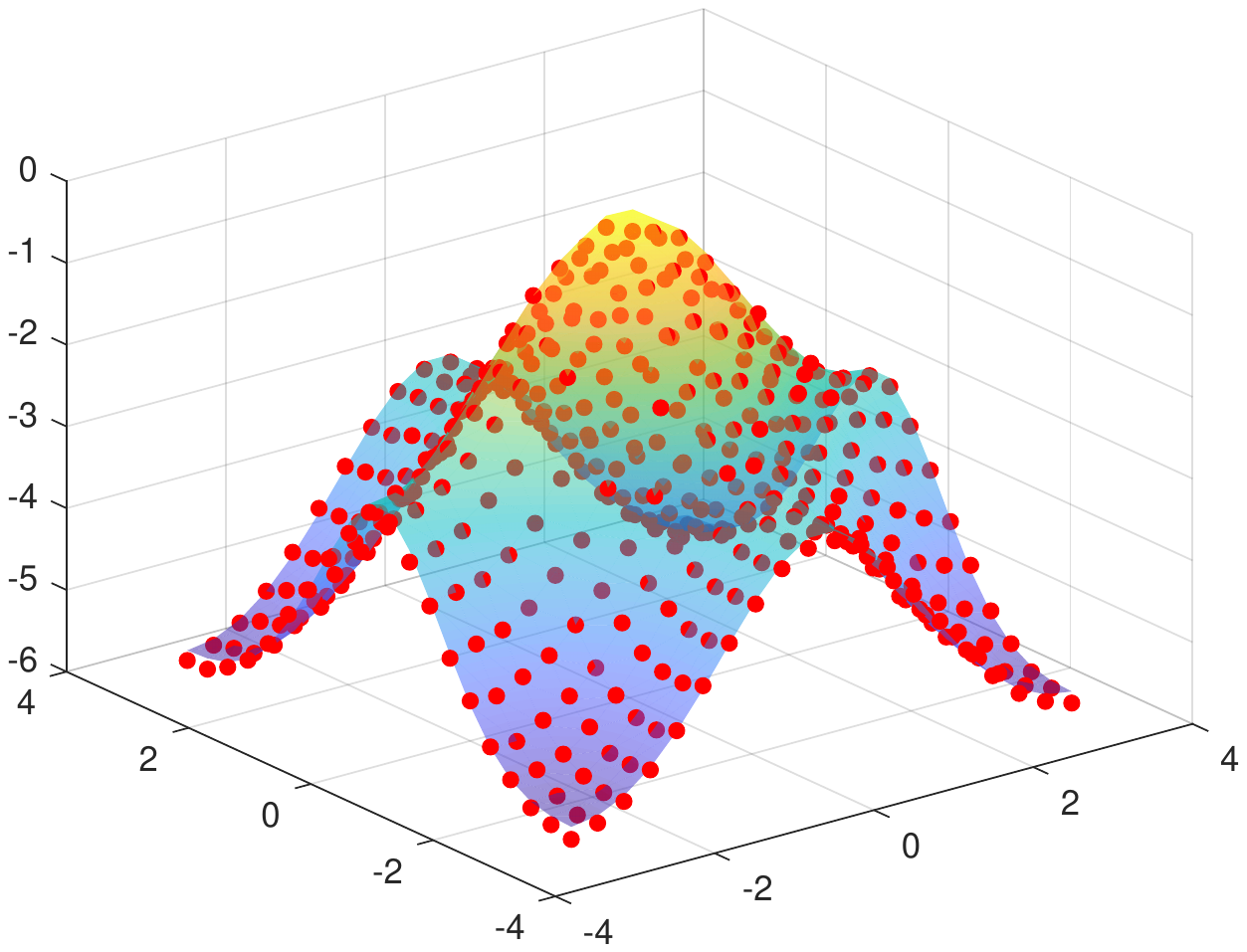}
				\subcaption{$\lambda_1(g)=-|f|$}
	\end{subfigure}
	\begin{subfigure}[b]{0.48\textwidth}
		\includegraphics[width = \textwidth,clip=true,trim = 0cm 7cm 0cm 7cm]{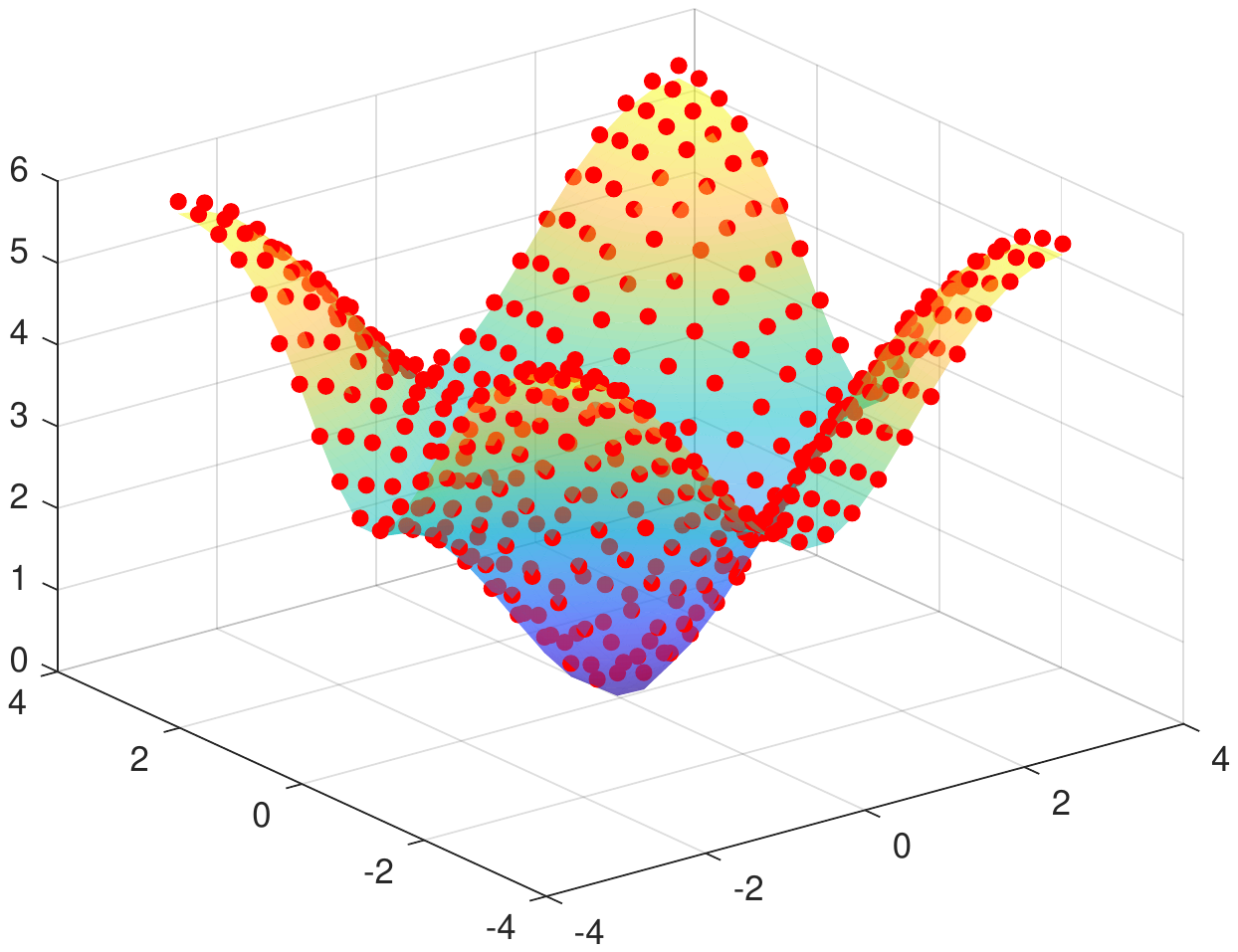}
						\subcaption{$\lambda_2(g)=|f|$}
	\end{subfigure}
	\caption{Eigenvalues of $Y_{\ven}T_{\ven}(f)$ (red dots) and the spectrum of $g$ (colored surfaces) for \cref{ex:2} when $n_1 = 20$ and $n_2 = 40$.} 
	\label{fig:ex2_surface}
\end{figure}

\begin{table}
\centering
\caption{Preconditioned MINRES iteration counts for \cref{ex:2}.}
\begin{tabular}{c c c c}
\hline
$\npr$ & $T_{\ven}(f_R)$ & $P_{\ven}^{2,2}$ & $P_{\ven}^{2,\beta}$\\
\hline
\hline
$10^2$ & 12 & 29 & 22\\
$20^2$ & 13 & 35 & 26\\
$40^2$ & 14 & 41 & 27\\
$80^2$ & 14 & 43 & 29\\
\hline
\end{tabular}\label{tab:ex2}
\end{table}

\end{example}

\begin{example} \label{ex:3}
In our final example we consider the $3$-level Toeplitz matrix arising from an upwind finite difference discretization of the convection-diffusion equation 
\begin{align*}
\begin{cases}
-\triangle u(x,y,z) + \boldsymbol{w}\cdot\nabla u(x,y,z)= f(x,y,z), \qquad \qquad &(x,y,z)\in\Omega,\\
u(x,y,z)=0,  &(x,y,z)\in\partial\Omega,
\end{cases}
\end{align*}
where $\Omega=(0,1)^3$ and $\boldsymbol{w} = [2, \ 1, \ 1.5]^T$.

For fixed $n_1,n_2,n_3\in\mathbb{N}$, we take the following equispaced partition of $\Omega$
\begin{eqnarray*}
	x_i=ih_x,&\quad h_x=\frac{1}{n_1+1},&\quad i=0,1,\dots,n_1,\\
	y_j=jh_y,&\quad h_y=\frac{1}{n_2+1},&\quad j=0,1,\dots,n_2, \\
	z_k=kh_z,&\quad h_z=\frac{1}{n_3+1},&\quad k=0,1,\dots,n_3, \\
\end{eqnarray*}
and apply the discretization in \cite{ChNg02}. The resulting coefficient matrix is 
$T_{\ven} = T_{n_3} \otimes I_{n_2} \otimes I_{n_1} + I_{n_3} \otimes T_{n_2}\otimes I_{n_1} + I_{n_3} \otimes I_{n_2} \otimes T_{n_1}$,  where $\ven = (n_1,n_2,n_3)$,  
\[
T_{n_1} = \begin{bmatrix} \mathfrak{a} & \mathfrak{c} & & & &  \\ \mathfrak{b} & \mathfrak{a} & \mathfrak{c} & & \\ & \ddots & \ddots & \ddots & \\ & & \mathfrak{b} & \mathfrak{a} & \mathfrak{c}\\ & & & \mathfrak{b} & \mathfrak{a}\end{bmatrix}, \quad 
T_{n_2} = \begin{bmatrix} 0 & \mathfrak{e} & & & &  \\ \mathfrak{d} & 0 & \mathfrak{e} & & \\ & \ddots & \ddots & \ddots & \\ & & \mathfrak{d} & 0 & \mathfrak{e}\\ & & & \mathfrak{d} & 0\end{bmatrix}
\]
\[T_{n_3} = \begin{bmatrix} 0 & \mathfrak{g} & & & &  \\ \mathfrak{f} & 0 & \mathfrak{g} & & \\ & \ddots & \ddots & \ddots & \\ & & \mathfrak{f} & 0 & \mathfrak{g}\\ & & & \mathfrak{f} & 0\end{bmatrix}
\]
with 
$\mathfrak{a} = 6 + 2 h_x + h_y + 1.5 h_z$, $\mathfrak{b} = -1 -2 h_x$, $\mathfrak{c} = -1$, $\mathfrak{d} = -1 - h_y$, $\mathfrak{e} = -1$, $\mathfrak{f} = -1 -1.5 h_z$, and $\mathfrak{g} = -1$. 
The associated symbol is $f(\theta_1,\theta_2,\theta_3) = f_1(\theta_1) + f_2(\theta_2) + f_3(\theta_3)$, where 
$f_1(\theta) = \mathfrak{a} + \mathfrak{b}{\rm e}^{\i \theta} + \mathfrak{c}{\rm e}^{-\i \theta}$, $f_2(\theta) = \mathfrak{d}{\rm e}^{\i \theta} + \mathfrak{e}{\rm e}^{-\i \theta}$ and $f_3(\theta) = \mathfrak{f}{\rm e}^{\i \theta} + \mathfrak{g}{\rm e}^{-\i \theta}$.

Also for this example we check the performance of the preconditioned MINRES method for solving the linear system $T_{\ven}u_{\ven}=b_{\ven}$ with $b_{\ven}=\ve{1}$. As preconditioners we choose $\mathcal{P}\sn=T_{\ven}(f_{R})$ and the positive definite 3-level circulant preconditioner defined as $$\mathcal{P}\sn=C_{\ven}=|C_{n_3}| \otimes I_{n_2} \otimes I_{n_1} + I_{n_3} \otimes |C_{n_2}| \otimes I_{n_1} + I_{n_3} \otimes I_{n_2} \otimes |C_{n_1}|$$ with $|C_{n_\ell}|=(C_{n_\ell}^T C_{n_\ell})^{\frac12}$, where $C_{n_\ell}$ is the optimal circulant preconditioner for $T_{n_\ell}$, with $\ell=1,2,3$. In the latter case, $h=|f_1|+|f_2|+|f_3|$. These symmetric positive definite preconditioners satisfy the conditions of \cref{thm:precflip}.

\cref{fig:ex3}(a)--(c) shows the matching between the eigenvalues of $Y_{\ven}T_{\ven}$ or $\mathcal{P}\sn^{-1}Y_{\ven}T_{\ven}$ and the sampling of the eigenvalue functions of $g$ or $h^{-1}g$ contained in $\Lambda$. From these pictures we infer that, as in previous example, $T_{\ven}(f_{R})$ is a good preconditioner. On the contrary, we expect that since $\lambda_1(h^{-1}g),\lambda_2(h^{-1}g)$ are not clustered away from 0, $C_{\ven}$ is not able to ensure fast convergence. This is confirmed by the iteration counts in \cref{tab:ex3}.

\begin{figure}[htbp]
	\centering
	\begin{subfigure}[b]{0.30\textwidth}
		\includegraphics[width = \textwidth,clip=true,trim = 0cm 7cm 0cm 7cm]{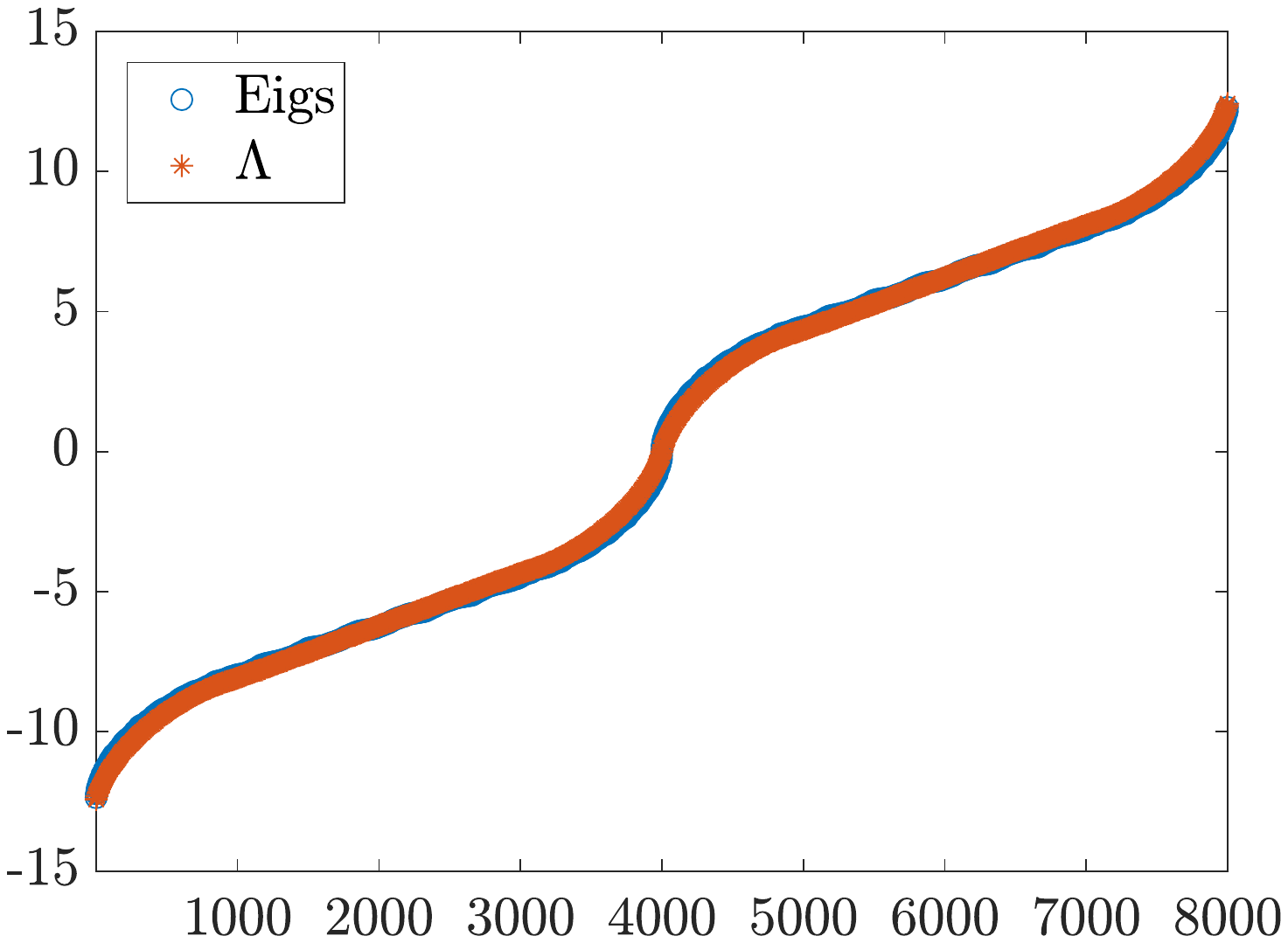}
		\subcaption{Unpreconditioned}
	\end{subfigure}
	\begin{subfigure}[b]{0.30\textwidth}
		\includegraphics[width = \textwidth,clip=true,trim = 0cm 7cm 0cm 7cm]{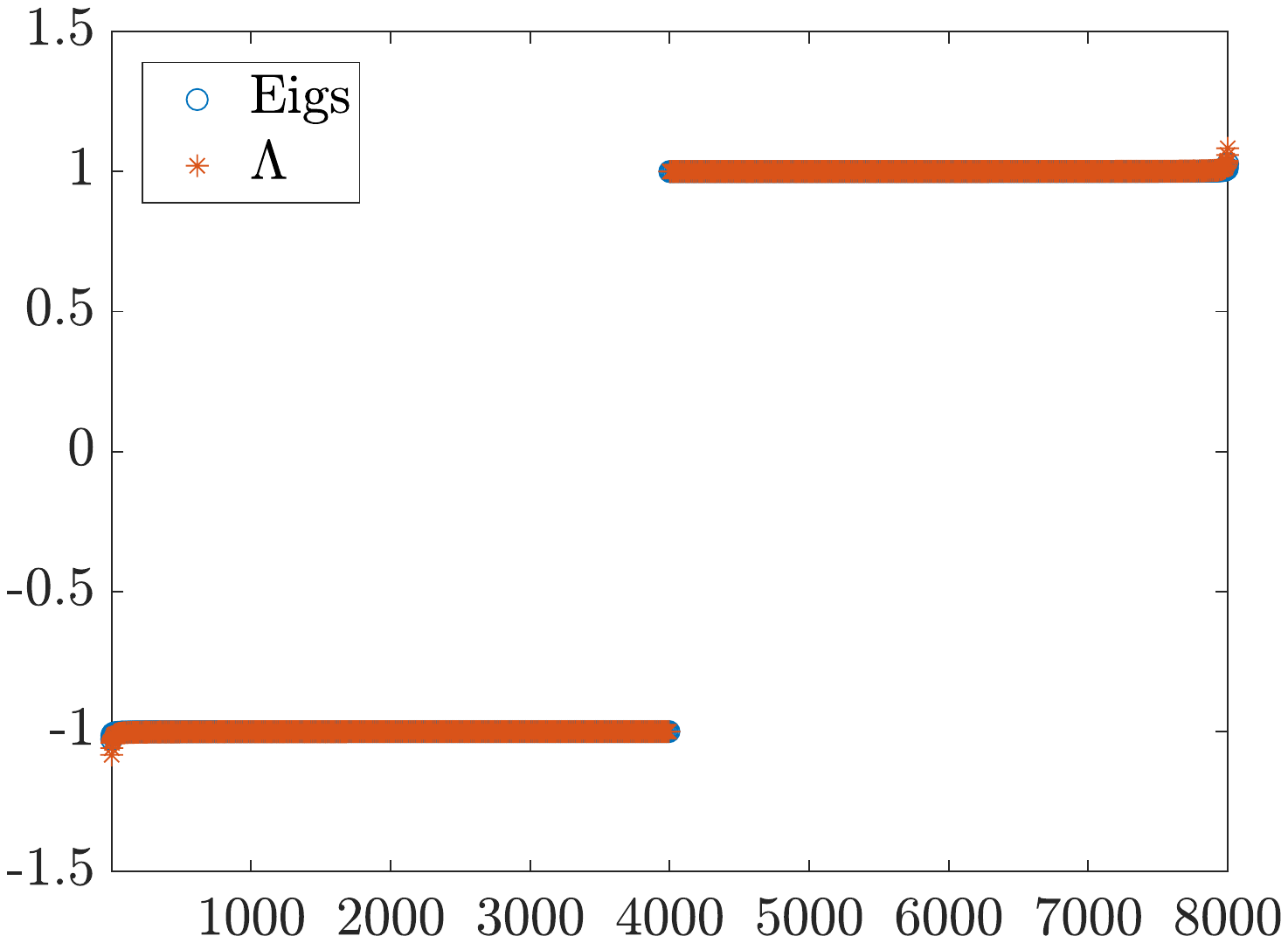}
		\subcaption{$\Prec\sn = T_{\ven}(f_R)$}
	\end{subfigure}
	\begin{subfigure}[b]{0.30\textwidth}
		\includegraphics[width = \textwidth,clip=true,trim = 0cm 7cm 0cm 7cm]{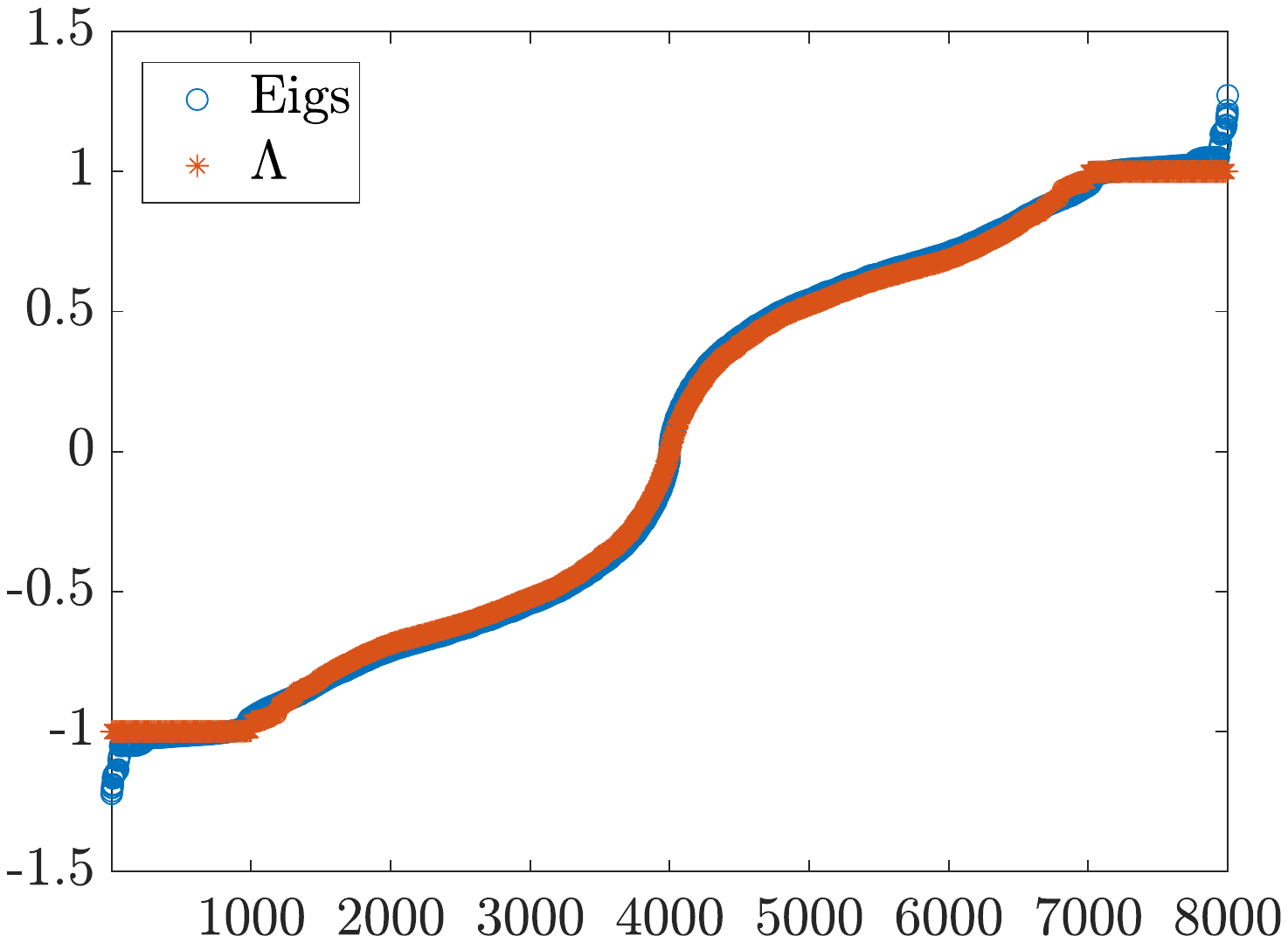}
		\subcaption{$\Prec\sn= C_{\ven}$}
	\end{subfigure}
	\caption{Comparison of the eigenvalues of $Y_{\ven}T_{\ven}$ or $\Prec\sn^{-1}Y_{\ven}T_{\ven}$ ({\color{blue} $\circ$}) with $\Lambda$ collecting the uniform samples of the eigenvalue functions of $g$ or $h^{-1}g$ for \cref{ex:3} ({\color{red}$\ast$}) $n_1=n_2=n_3= 20$.} 
	\label{fig:ex3}
\end{figure}

\begin{table}
	\centering
	\caption{Preconditioned MINRES iteration counts for \cref{ex:3}.}
	\begin{tabular}{c c c}
\hline
$\npr$ & $T_{\ven}(f_R)$ & $C_{\ven}$ \\
\hline
\hline
$5^3$ & 8 & 61\\
$10^3$ & 9 & 198\\
$20^3$ & 9 & 724\\
\hline
	\end{tabular}\label{tab:ex3}
\end{table}

\end{example}

\section{Conclusions} \label{sec:conc}
We have shown that the asymptotic eigenvalue distribution of $\{\Yn\Tn(f)\}\sn$, 
where $\Tn(f)$ is a square real multilevel Toeplitz matrix generated by $f\in L^1([-\pi,\pi]^d)$ and 
$\Yn$ is the exchange matrix, is governed by a $2\times 2 $ matrix-valued function 
whose eigenvalues are $\pm |f|$. We have also investigated the asymptotic eigenvalue distribution of 
preconditioned sequences  $\{\mathcal{P}^{-1}\sn \Yn\Tn(f)\}\sn$, where  
 $\mathcal{P}\sn$ is Hermitian positive definite,  
$\{\Prec\p\}\p\sim_{\rm GLT}h$, and $\{\Pin\Un\Prec\p\Un\Pin^T\}\p\sim_{\rm GLT}h\cdot I_2$ with $h:[-\pi,\pi]^d\rightarrow\mathbb{C}$ and $h\ne0$ a.e.
The latter result enables us to analyse the convergence of preconditioned MINRES for this problem at least in the two quite common cases where the preconditioners are multilevel circulant or multilevel Toeplitz matrices.

\section*{Acknowledgement}
The first author is member of the INdAM research group
GNCS and her work was partly supported by the GNCS-INDAM Young Researcher Project 2020 titled \lq\lq Numerical methods for image restoration and cultural heritage deterioration''. 
The second author gratefully acknowledges support from the EPSRC grant EP/R009821/1. No new data were created for this publication.

\bibliographystyle{siamplain}
\bibliography{Hankel_abbrv}
\end{document}